\documentclass[preprint,reqno,12pt]{amsart}
\usepackage{amssymb}
\usepackage{amsmath, amsfonts, amsthm, cite}
\usepackage{color}
\input xy
\xyoption{all}

\usepackage{eepic}
\usepackage{amssymb}
\usepackage{amsmath}
\usepackage{amsfonts}
\usepackage{indentfirst}
\usepackage{a4wide}
\usepackage{amstext}
\usepackage{amsthm}

\newtheorem{theorem}{Theorem}[section]

\newtheorem{lemma}[theorem]{Lemma}
\newtheorem{prop}[theorem]{Proposition}
\newtheorem{cor}[theorem]{Corollary}

\theoremstyle{definition}
\newtheorem{defn}[theorem]{Definition}

\newtheorem{example}[theorem]{Example}

\theoremstyle{remark}

\newtheorem{remark}[theorem]{Remark}

\numberwithin{equation}{section}

\def\U{{\mathcal U}}

\def\R{{\mathbb R}}
\def\T{{\mathbb T}}
\def\t{{\mathfrak{t}}}

\def\K{{\mathcal K}}
\def\W{{\mathcal W}}

\def\Z{{\mathbb Z}}

\begin{document}

\baselineskip=1.2\baselineskip
 
\title[A \v{C}ech Dimensionally Reduced Gysin Sequence for Principal Torus Bundles]{A \v{C}ech 
Dimensionally Reduced Gysin Sequence\\ for Principal Torus Bundles}

\author[P Bouwknegt]{Peter Bouwknegt}

\address[P Bouwknegt]{Department of Mathematics, Mathematical Sciences Institute, and
Department of Theoretical Physics, Research School of Physics and Engineering,
Australian National University, Canberra ACT~0200, Australia}
\email{peter.bouwknegt@anu.edu.au}

\author[R Ratnam]{Rishni Ratnam}

\address[R Ratnam]{Department of Mathematics, 
Mathematical Sciences Institute, Australian National University, Canberra
ACT~0200, Australia}
\email{rishni.ratnam@anu.edu.au, rishniratnam@yahoo.com.au}

\thanks{This research was supported under the Australian Research Council's Discovery
Projects funding scheme (project numbers DP0559415 and DP0878184).}

\begin{abstract}
In this paper we construct \v{C}ech cohomology groups that form a Gysin-type long exact 
sequence for principal torus bundles. This sequence is modeled on a de Rham cohomology 
sequence published in earlier work by Bouwknegt, Hannabuss and Mathai, which was 
developed to compute the global properties of T-duality in the presence of NS H-Flux.
\end{abstract}

\maketitle


\section{Introduction}

In \cite{BouHanMat05} the authors formulate a ``dimensionally reduced" Gysin sequence for de 
Rham cohomology. From their point of view, the purpose of this sequence is to compute the global 
properties of T-duality for principal torus bundles with background NS H-flux (see, e.g., 
\cite{BEMa, BEMb, BHMa, BunSch05, BunRumSch06, MatRos05, MatRos06}).   This Gysin sequence 
utilises a Chern-Weil or ``dimensional reduction" isomorphism, allowing a differential form on the total 
space of a principal torus bundle to be expressed as a tuple of forms over the base manifold. 
T-duality is then easily computable by concatenating or truncating tuples of forms corresponding 
to the bundle curvature and flux. In this paper, we extend much of \cite{BouHanMat05} to \v{C}ech 
cohomology, thus including the phenomena of torsion. To begin with, we shall motivate the use of Gysin 
sequences in T-duality.

Recall that, when working with principal $\T$-bundles, T-duality is an order two transformation 
$T$ defined on the set of pairs $(c,\delta)$, where $c\in \check{H}^2(Z,\underline{\Z})$ is the 
Euler class of a principal $\T$-bundle $\pi:X\to Z$ and $\delta\in \check{H}^3(X,\underline{\Z})$
represents the H-flux.
The image of $T$ is a pair $(\hat{c},\hat{\delta})$, consisting of an Euler class of some principal 
$\T$-bundle $\hat{\pi}:\hat{X}\to Z$ together with a class $\hat{\delta}\in 
\check{H}^3(\hat{X},\underline{\Z})$  
called the \emph{T-dual} Euler class and H-flux, respectively \cite{BEMa, BEMb}.
The T-dual Euler class $\hat{c}$ can be 
obtaineded from the Gysin sequence:
\begin{equation}\label{circlegysin}
\ldots \to \check{H}^k(Z,\underline{\Z})\stackrel{\pi^*}{\to} \check{H}^k(X,\underline{\Z})\stackrel{\pi_*}{\to} 
\check{H}^{k-1}(Z,\underline{\Z})\stackrel{\cup c }{\to} \check{H}^{k+1}(Z,\underline{\Z})\to \ldots \,,
\end{equation}
by defining
\begin{equation}\label{T1}
\hat{c}:=\pi_{*}\delta,
\end{equation}
so that there is a principal $\T$-bundle $\hat{\pi}:\hat{X}\to Z$ classified by $\hat{c}$. 
Consider now the Gysin sequence corresponding to the T-dual Euler class:
\begin{equation}\label{circlegysindual}
\ldots \to \check{H}^k(Z,\underline{\Z})\stackrel{\hat{\pi}^*}{\to} \check{H}^k(\hat{X},\underline{\Z})
\stackrel{\hat{\pi}_*}{\to} \check{H}^{k-1}(Z,\underline{\Z})
\stackrel{\cup \hat{c}}{\to} \check{H}^{k+1}(Z,\underline{\Z})\to \ldots \,.
\end{equation}
Observe that the sequence (\ref{circlegysin}) implies $c$ satisfies
\begin{align*}
c\cup \hat{c}=&\hat{c}\cup c = \pi_{*}\delta\cup c=0.
\end{align*}
Therefore, exactness of the sequence (\ref{circlegysindual}) implies there 
exists a class $\hat{\delta}$ such that
\begin{equation}\label{T2}
\hat{\pi}_*\hat{\delta}=c.
\end{equation}
If $\hat{\delta}$ is to be the dual H-flux, then for physical reasons, one must also 
consider the correspondence diagram:

\centerline{ \xymatrix{
&X\times_Z\hat{X} \ar[dl]_{p}\ar[dr]^{\hat{p}} & \\
X\ar[dr]_{\pi} & & \hat{X}\ar[dl]^{\hat{\pi}}\\
& Z &}} \

\noindent and the requirement that
\begin{align}\label{T3}
p^*\delta&=\hat{p}^*\hat{\delta}.
\end{align}
This condition says the component of the H-flux living on the base 
manifold $Z$ is left untouched by the duality.

The equations (\ref{T1}), (\ref{T2}) and (\ref{T3}) uniquely specify the T-dual Euler class, but not the 
T-dual H-flux. In fact, chasing through the diagrams (c.f. \cite[Sect 2.2]{BunSch05}) shows that 
$\hat{\delta}$ is only determined up to addition of a class of the form $\hat{\pi}^*([L]\cup c)$, where 
$[L]\in\check{H}^1(Z,\underline{\Z})$. On the other hand, one can show \cite[Thm 2.16]{BunSch05}
that for any choice of $[L]$ there is an automorphism of $\hat{\pi}:X\to Z$ such that 
$\hat{\delta}+\hat{\pi}^*([L]\cup c)\mapsto \hat{\delta}$. Therefore, within the isomorphism class 
$\hat{\pi}:X\to Z$, there is an unique choice of T-dual H-flux $\hat{\delta}$.

To discuss the higher dimensional case, we introduce some terminology.

\begin{defn}
Let $\pi:X\to Z$ be a principal $\T^n$-bundle, and suppose it is classified by a 
class $c\in\check{H}^2(Z,\underline{\Z}^n).$ Then we call $c$ the \emph{Euler vector} of $\pi:X\to Z$.
\end{defn}

\begin{defn}
Let $(c,\delta)\in\check{H}^2(Z,\underline{\Z}^n)\oplus \check{H}^3(X,\underline{\Z})$ be a pair of 
cohomology classes, where $c$ is the Euler vector of a principal $\T^n$-bundle $\pi:X\to Z$. 
Then we call $(c,\delta)$ a \emph{T-duality pair}.
\end{defn}

Thus, in the general case of a T-duality pair $(c,\delta)\in\check{H}^2(Z,\underline{\Z}^n)\oplus
\check{H}^3(X,\underline{\Z})$ we seek to create a Gysin sequence generalising (\ref{circlegysin}) 
and (\ref{circlegysindual}) so that we may compute T-duality using 
analogues of (\ref{T1}),(\ref{T2}) and (\ref{T3}):

\begin{align}
\label{bhmgysin1}\hat{c}&=\pi_{*}\delta,\\
\label{bhmgysin2}c&=\hat{\pi}_*\hat{\delta},\\
\label{bhmgysin3}p^*\delta&=\hat{p}^*\hat{\delta}.
\end{align}

This was partially done in \cite{BouHanMat05}, where a higher dimensional analgue of (\ref{circlegysin}) 
was computed in de Rham cohomology. We recall this generalisation presently. Let $Z$ be a $C^\infty$ 
manifold, and denote by $H^k_{dR}(Z,\t)$ the $k^{th}$ de Rham cohomology group with values in the Lie 
algebra $\t$ of $\T^n$, which we identify with $\R^n$. Now, suppose $\pi:X\to Z$ is our principal 
$\R^n/\Z^n$-bundle classified, over some open cover $\W$ of $Z$, by 
$[F]\in \check{H}^2(\W,\underline{\Z}^n)$. 
If $\underline{\R}^n$ denotes the constant sheaf with values in $\R^n$, then there is an obvious 
map $\check{Z}^k(\W,\underline{\Z}^n)\to \check{Z}^k(\W,\underline{\R}^n)$, so that, if $\Omega^k(Z,\R^n)$ 
denotes the ring of differential $k$-forms with values in $\R^n$, composition with the \v{C}ech-de Rham 
map $\check{Z}^k(\W,\underline{\R}^n)\to \Omega^k(Z,\R^n)$ (which is an isomorphism on cohomology) 
gives a map $\check{Z}^k(\W,\underline{\Z}^n)\to \Omega^k(Z,\R^n).$ The image of the cocycle $F$ in 
$\Omega^2(Z,\R^n)$ is a closed form, denoted $F_2$, which we call a \emph{curvature} form of $\pi:X\to Z$ 
(we give an explicit formula for $F_2$ in terms of $F$ in Equation (\ref{F2rep}) below).

Let $\wedge^i \t^*$ denote the $i^{th}$ exterior power of the dual 
Lie algebra $\t^*$. For $0\leq m\leq l\leq k$ 
we define a cochain complex
\begin{equation}\label{BHMComplex}
C_{F_2}^{k,(m,l)}(Z,\t^*):=\bigoplus_{i=m}^l\Omega^{k-i}(Z,\wedge^i \t^*),
\end{equation}
where $\Omega^{k-i}(Z,\wedge^i \t^*)$ is the group of differential $(k-i)$-forms with values 
in $\wedge^i \t^*$. Let $\{X_i\}_{i\in (1,\dots ,n)}$ and $\{X_i^*\}_{i\in (1,\dots ,n)}$ be a basis 
of $\mathfrak t$, and dual basis of $\mathfrak t^*$, respectively. 
Note that $\{X_{j_1}^*\wedge X_{j_2}^*\wedge \dots 
\wedge X_{j_i}^*\,:0\leq j_1<j_2<\dots < j_i\leq n\}$ 
is a basis for $\wedge^i \t^*$, so that every element $H_{(k-i)i}$ of 
$\Omega^{k-i}(Z,\wedge^i \t^*)$ can be written as a sum
\[\sum_{1\leq j_1<j_2<\dots < j_i\leq n}(\tilde{H}_{(k-i)i})_{j_1\dots j_i}\otimes 
X_{j_1}^*\wedge\dots\wedge X_{j_i}^*,\]
with $(\tilde{H}_{(k-i)i})_{j_1\dots j_i}\in\Omega^{k-i}(Z)$. To define a differential on the cochain 
complex (\ref{BHMComplex}), we first observe that on each group 
$\Omega^{k-i}(Z,\wedge^i \t^*)$ we can define maps
\[\wedge F_2:\Omega^{k-i}(Z,\wedge^i \t^*)\to \Omega^{k-i+2}(Z,\wedge^{i-1} \t^*)\]
by the formula
\begin{align*}
&H_{(k-i)i}\wedge F_2:=\\
&\sum_{1\leq j_1<j_2<\dots < j_i\leq n}\sum_{l=1}^i(-1)^{l+1}(\tilde{H}_{(k-i)i})_{j_1\dots j_i}\wedge 
F_2(X_{j_l}^*)\otimes X_{j_1}^*\wedge\dots \wedge\widehat{X_{j_l}^*}\wedge\dots\wedge X_{j_i}^*,\\
\end{align*}
where $\widehat{X_{j_l}^*}$ denotes omission of $X_{j_l}^*$. Then we have a differential 
$D_{F_2}$ on  $C_{F_2}^{k,(m,l)}(Z,\t^*)$
given by
\begin{align*}
&D_{F_2}(H_{(k-m)m},\dots,H_{(k-l)l})=(dH_{(k-m)m}+(-1)^{k-m-1}H_{(k-m-1)(m+1)}\wedge F_2,\\
&\quad dH_{(k-m-1)(m+1)}+(-1)^{k-m-2}H_{(k-m-2)(m+2)}\wedge F_2, \dots,dH_{(k-l)l}).
\end{align*}
The fact that $F_2(X^*)\wedge F_2(Y^*)=F_2(Y^*)\wedge F_2(X^*)$ for all dual vectors $X^*$ 
and $Y^*$ implies $D_{F_2}^2=0$, and we denote the resulting cohomology groups 
by $H^{k,(m,l)}_{F_2}(Z,\t^*)$. This differential can be seen diagrammatically in Figure \ref{BHMdifferential}.

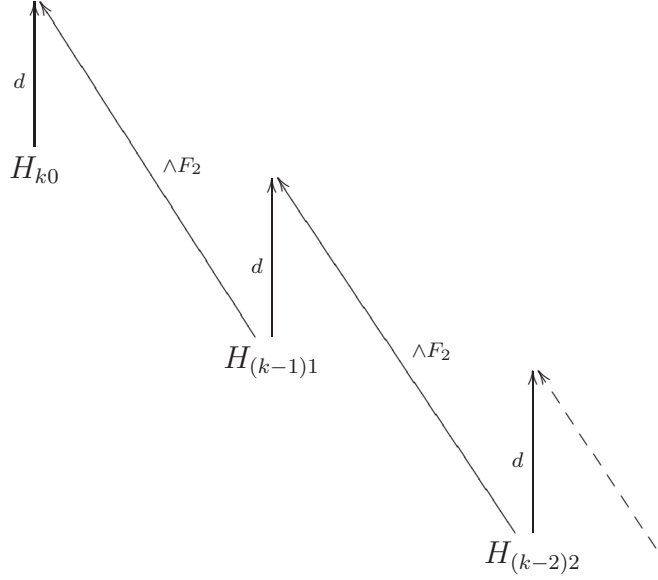
\begin{figure}
\centerline{\xymatrix{
&&&&&\\
&&&&&\\
H_{k0}\ar[uu]^{d}&&&&&\\
&&&&&\\
&&H_{(k-1)1}\ar[uu]^{d}\ar[uuuull]_{\wedge F_2}&&&\\
&&&&&\\
&&&&H_{(k-2)2}\ar[uu]^{d}\ar[uuuull]_{\wedge F_2}& \ar@{-->}[uul]
}}
\caption{The differential $D_{F_2}$}
\label{BHMdifferential}
\end{figure}

\begin{example}\label{l=2example}
Set $m=0$ and $l=2$, and consider a $k$-cochain $(H_{k0},H_{(k-1)1}$, 
$H_{(k-2)2})$ in $C_{F_2}^{k,(0,2)}(Z,\t^*)$. Let 
\begin{align*}
H_{(k-2)2}=&\sum_{1\leq i<j\leq n}(\tilde{H}_{(k-2)2})_{ij}\otimes X_i^*\wedge X_j^*,\quad \mbox{and}\\
H_{(k-1)1}=&\sum_{1\leq i\leq n}(\tilde{H}_{(k-1)1})_{i}\otimes X_i^*.\\
\end{align*}
Thus $H_{(k-2)2}\wedge F_2$ is a differential $k$-form with values in $\t^*\cong\R^n$.  Let us define
\[e_l:=(0,\dots,\underbrace{1}_{l^{th} \mbox{ entry}},\dots,0).\]
Then, the $l^{th}$ component $(H_{(k-2)2}\wedge F_2)_l$ can be 
found by computing $(H_{(k-2)2}\wedge F_2)(X_l)$, so one can see
\begin{align}
(H_{(k-2)2}\wedge F_2)_l=&\sum_{1\leq i<j\leq n}(\tilde{H}_{(k-2)2})_{ij}
\wedge F_2(X_i^*)\otimes X_j^*(X_l)\nonumber \\
&\quad-(\tilde{H}_{(k-2)2})_{ij}\wedge F_2(X_j^*)\otimes X_i^*(X_l) \notag\\
\label{wedgef1}=&\sum_{1\leq i<j\leq n}(\tilde{H}_{(k-2)2})_{ij}\wedge 
\big[(F_2)_i(e_l)_j-(F_2)_j (e_l)_i\big].
\end{align}
Similarly we have
\begin{equation}\label{wedgef2}
H_{(k-1)1}\wedge F_2=\sum_{i=1}^n(\tilde{H}_{(k-1)1})_i\wedge (F_2)_i \,.
\end{equation}
Therefore, $(H_{k0},H_{(k-1)1},H_{(k-2)2})$ is a cocycle if all of the following hold:
\begin{align*}
&dH_{k0} +(-1)^{k-1}\sum_{i=1}^n(\tilde{H}_{(k-1)1})_i\wedge (F_2)_i=0,\\
&(dH_{(k-1)1})_l+ (-1)^{k-2}\sum_{1\leq i<j\leq n}(\tilde{H}_{(k-2)2})_{ij}\wedge 
\big[(F_2)_i(e_l)_j-(F_2)_j (e_l)_i\big]=0,\\
&dH_{(k-2)2}=0.
\end{align*}
\end{example}

\begin{remark}\label{dependsonF2}
The above definition of $D_{F_2}$ differs from that of \cite{BouHanMat05} by 
our convention of minus signs. These are more convenient for showing the 
relationship between their Gysin sequence and ours.
\end{remark}

\begin{remark}
Notice that the above definition for $H^{k,(m,l)}_{F_2}(Z,\t^*)$ depends on 
the choice of representative $F_2$.
\end{remark}

The first reason why the groups $H^{k,(m,l)}_{F_2}(Z,\t^*)$ are of interest is 
the following theorem, which gives a ``dimensional reduction" isomorphism for de 
Rham cohomology:

\begin{theorem}[{\cite[Sect 3.1]{BouHanMat05}}]\label{ChernWeil}
Let $Z$ be a $C^\infty$ manifold, and suppose $\pi:X\to Z$ is a principal $\R^n/\Z^n$-bundle with 
curvature $F_2\in \Omega^{2}(Z,\t)$. Then there exists an isomorphism
\[H^k_{dR}(X)\cong H^{k,(0,k)}_{F_2}(Z,\t^*).\]
\end{theorem}

This isomorphism depends on a choice of principal connection with curvature $F_2$, and therefore is 
not canonical. Moreover, the groups $H^{k,(m,l)}_{F_2}(Z,\t^*)$ fit into a Gysin sequence:
\begin{theorem}[{\cite[Thm 3.2]{BouHanMat05}}]\label{BHMGysin}
There exists an exact sequence
\begin{equation}
\ldots \to H^{k,(m,m)}_{F_2}(Z,\t^*)\stackrel{\pi^*}{\to} H^{k,(m,l)}_{F_2}(Z,\t^*)\stackrel{\pi_*}{\to} 
H^{k,(m+1,l)}_{F_2}(Z,\t^*)\stackrel{\wedge F_2}{\to} H^{k+1,(m,m)}_{F_2}(Z,\t^*)\to \ldots \,,
\end{equation}
where the map $\wedge F_2:H^{k,(m+1,l)}_{F_2}(Z,\t^*)\to H^{k+1,(m,m)}_{F_2}(Z,\t^*)$ is
\[[(H_{(k-m-1)(m+1)},\dots,H_{(k-l)l})]\wedge F_2:= [(-1)^{k-m}H_{(k-m-1)(m+1)}\wedge F_2].\]
\end{theorem}

This sequence is called a Gysin sequence because, in the case $m=0$ and $l=k$, the sequence 
corresponding to a principal $\T$-bundle $X\to Z$ is identical to the image in de Rham cohomology 
of (\ref{circlegysin}).

Suppose then that we have a T-duality pair $(c,\delta)\in \check{H}^2(Z,\underline{\Z}^n)\oplus 
\check{H}^2(X,\underline{\Z})$ such that that $c$ and $\delta$ have images $[F_2]\in 
H^2_{dR}(Z,\mathfrak{t})$ and $[H]\in H^3_{dR}(X)$ under the compositions
\begin{align*}
&\check{H}^2(Z,\underline{\Z}^n)\to \check{H}^2(Z,\underline{\R}^n)\to H^2_{dR}(Z,\mathfrak{t})\,,\\
&\check{H}^3(Z,\underline{\Z})\to \check{H}^3(Z,\underline{\R})\to H^3_{dR}(Z) \,,
\end{align*}
respectively. We then let $[H_3^{dR},H_2^{dR},H_1^{dR},H_0^{dR}]\in 
H^{3,(0,3)}_{F_2}(Z,\mathfrak{t}^*)$ 
be the image of $[H]$ under the isomorphism from Theorem \ref{ChernWeil}, and by 
Equation (\ref{bhmgysin1}) 
the \emph{T-dual curvature} is the class $[H_2^{dR},H_1^{dR},H_0^{dR}]\in 
H^{3,(1,3)}_{F_2}(Z,\mathfrak{t}^*)$. 
Notice that \cite{BouHanMat05} does not provide a dual Gysin sequence in the case 
$H_1^{dR}\neq 0,H_0^{dR}\neq 0$ (i.e.\ an analogue of (\ref{circlegysindual})). 
Thus, the T-dual H-flux, which should be computed by (\ref{bhmgysin2}) and (\ref{bhmgysin3}), 
strictly speaking, is  unknown, although we 
expect it in some sense to have a representative of the form $(H_3^{dR},F_2,0,0)$.

In the general case it was proposed in \cite{BouHanMat05, BHMb, BHMc} that 
$[H_2^{dR},H_1^{dR},H_0^{dR}]$ defines 
a ``nonassociative torus bundle", and in the case $H_0^{dR}=0$, the class $[H_2^{dR},H_1^{dR}]\in 
H^{3,(1,2)}_{F_2}(Z,\mathfrak{t}^*)$ defines a noncommutative torus bundle \cite{MatRos05, MatRos06}. 
The links 
between nonassociative/noncommutative bundles and the tuples of differential forms 
was however not made 
precise. Moreover, since this previous work was done in de Rham cohomology, it neglected the 
phenomena of torsion. In this, and a companion paper \cite{BouCarRat2} (see also \cite{Rat}
for more details and \cite{BouCarRat1} 
for a review and additional examples), we shall rectify some of 
these details in the case $H_0=0$ by constructing an \emph{integer} cohomology group analogous to 
$H^{3,(1,3)}_{F_2}(Z,\mathfrak{t}^*)$, and linking more directly the class $[H_2^{dR},H_1^{dR}]$ to a 
noncommutative torus bundle.

\section{Dimensionally Reduced \v{C}ech Cohomology}

With the previous section as motivation, we shall proceed as follows. First, we shall construct groups 
analogous to $H^{k,(m,l)}_{F_2}(Z,\t^*)$ using \v{C}ech cochains over a good open cover $\mathcal{W}$ 
of $Z$. Instead of using the curvature form $F_2$, we shall use a representative 
$F\in\check{Z}^2(\W,\underline{\Z}^n)$ of the Euler vector of $\pi:X\to Z$. Due the applications 
we have in mind, we shall restrict to the case where $m=0$ and $l=2$, and only construct an 
analogue of $H^{k,(0,2)}_{F_2}(Z,\t^*)$. It is in fact possible to show, with great difficulty, that the 
resulting cohomology groups are, up to isomorphism, independent of the choice of representative of 
$[F]\in \check{H}^2(\W,\underline{\Z}^n)$. For a proof of this fact, the reader should consult \cite{Rat}. 
Note however that the isomorphism given in \cite{Rat} is not canonical. Moreover, it is possible to show 
that when considering refinements, even in the case where different refinement maps induce 
\emph{identical} 
representatives of the Euler vector, the induced refinement maps have images that differ by a non-trivial 
automorphism. It follows that taking a colimit will be inappropriate for our applications. Thus, we are forced 
to fix a cover when using these groups here and in \cite{BouCarRat2}. Finally, we shall provide Gysin 
sequences for our \v{C}ech cohomology groups, and show they match with the Gysin sequence of 
\cite{BouHanMat05}. Note that the analogue of Theorem \ref{ChernWeil} is the 
primary focus of \cite{BouCarRat2}.

The main obstruction to constructing a direct \v{C}ech analogue of $H^{k,(m,l)}_{F_2}(Z,\t^*)$ is the 
fact that the cup product of \v{C}ech cocycles is not commutative. Therefore, simply abstracting the 
differential $D_{F_2}$ by using the cup product with $F$ instead of the wedge product with $F_2$ will 
not give a map that squares to zero (cf. the comments preceeding Example \ref{l=2example}). We shall 
surmount this problem by using the fact that the cup product is commutative on 
\v{C}ech \emph{cohomology}, 
which implies that commuted cocycles differ by a coboundary. This coboundary shall be incorporated 
into our analogue of $D_{F_2}$, giving a genuine differential.

\begin{lemma}[{\cite[Sect 2]{Ste47}}]\label{commutecoboudary}
Let $Z$ be a topological space with an open cover $\W$ and $A,B\in \check{Z}^2(\W,\underline{\Z})$. Then
\[A\cup B-B\cup A=\check\partial C\,, \]
where
\[C_{\lambda_0\lambda_1\lambda_2\lambda_3}(z):= A_{\lambda_0\lambda_1\lambda_2}(z)
B_{\lambda_0\lambda_2\lambda_3}(z)-A_{\lambda_1\lambda_2\lambda_3}
(z)B_{\lambda_0\lambda_1\lambda_3}(z).\]
\end{lemma}

Let $Z$ be a $C^\infty$ manifold and fix an open cover $\mathcal{W}=\{W_{\mu_0}\}$ of $Z$ together 
with a cocycle $F\in \check{Z}^2(\mathcal{W},\underline{\Z}^n)$. Let $l\in\Z$, and denote by 
$\mathcal{T}^{l}$ the sheaf of germs of continuous $\T^{n\choose l}$-valued functions. I.e.\
we can think of $\mathcal T^0$ and $\mathcal T^1$ as the sheafs $\mathcal S$ and $\hat{\mathcal N}$
of germs of continuous $\T$ and $\hat\Z^n$-valued functions, respectively.  Similarly, if $M_n^u(\T)$ 
denotes the group (under addition) of strictly upper triangular $n\times n$ matrices, we can view 
$\mathcal M:=\mathcal{T}^{2}$ as the sheaf of germs of continuous $M_n^u(\T)$-valued 
functions. Therefore 
we can think of a \v{C}ech cocycle  $\phi^{(k-2)2}\in \check{C}^{k-2}(\W,\mathcal{T}^{2})$ 
(respectively $\phi^{(k-1)1}\in\check{C}^{k-1}(\W,\mathcal{T}^{1})$) as an $n \choose 2$-tuple
\[\{\phi^{(k-2)2}(\cdot)_{ij}\}_{1\leq i<j\leq n},\]
where $\phi^{(k-2)2}(\cdot)_{ij}\in \check{C}^{k-2}(\W,\mathcal{T}^{0})$  (respectively, an $n$-tuple 
$\{\phi^{(k-1)1}(\cdot)_{i}\}_{1\leq i\leq n}$, with $\phi^{(k-1)1}(\cdot)_{i}\in 
\check{C}^{k-1}(\W,\mathcal{T}^{0})$).

We define a cochain complex $C^k_F(\mathcal{W},\mathcal{S})$, where for $k\geq 2$ an element 
is a triple $(\phi^{k0},\phi^{(k-1)1},\phi^{(k-2)2})$ consisting of \v{C}ech cochains $\phi^{k0}\in 
\check{C}^{k}(\W,\mathcal{T}^{0})$, $\phi^{(k-1)1}\in \check{C}^{k-1}(\W,\mathcal{T}^{1})$ and 
$\phi^{(k-2)2}\in \check{C}^{k-2}(\W,\mathcal{T}^{2})$. When $k=1$, we define a cochain 
to be a pair $(\phi^{10},\phi^{01
})$, where $\phi^{10}\in \check{C}^{1}(\W,\mathcal{T}^{0})$, $\phi^{01}\in 
\check{C}^{0}(\W,\mathcal{T}^{1})$, 
whilst when $k=0$ a cochain is a singleton $(\phi^{00})$, for $\phi^{00}\in 
\check{C}^{0}(\W,\mathcal{T}^{0})$.

Now, for any $A\in \check{C}^2(\W,\underline{\Z}^n)$ and $B\in \check{C}^3(\W,\underline{M_n^u(\Z)})$ 
we can define products
\begin{align*}
\cup_1 A&: \check{C}^{k-1}(\W,\mathcal{T}^1)\to \check{C}^{k+1}(\W,\mathcal{T}^0),\\
\cup_1 A&: \check{C}^{k-2}(\W,\mathcal{T}^2) \to \check{C}^{k}(\W,\mathcal{T}^1),\quad \mbox{and}\\
\cup_2 B&: \check{C}^{k-2}(\W,\mathcal{T}^2) \to \check{C}^{k+1}(\W,\mathcal{T}^0),
\end{align*}
with the formulas
\begin{align*}
&(\phi^{(k-1)1}\cup_1 A)_{\lambda_0\dots\lambda_{k+1}}(z):=
\prod_{1\leq l \leq n}\phi^{(k-1)1}_{\lambda_0\dots\lambda_{k-1}}(z)_l^{A_{\lambda_{k-1}\lambda_k
\lambda_{k+1}}(z)_l},\\
&(\phi^{(k-2)2}\cup_1 A)_{\lambda_0\dots\lambda_{k+1}}(z)_l:=\prod_{1\leq i<j\leq n}
\phi^{(k-2)2}_{\lambda_0\dots
\lambda_{k-1}}(z)_{ij}^{A_{\lambda_{k-1}\lambda_k\lambda_{k+1}}(z)_i(e_l)_j-
(e_l)_iA_{\lambda_{k-1}\lambda_k
\lambda_{k+1}}(z)_j},\\
&(\phi^{(k-2)2}\cup_2 B)_{\lambda_0\dots\lambda_{k+1}}(z):=\prod_{1\leq i<j\leq n}
\phi^{(k-2)2}_{\lambda_0
\dots\lambda_{k-2}}(z)_{ij}^{B_{\lambda_{k-2}\lambda_{k-1}\lambda_{k}\lambda_{k+1}}(z)_{ij}}.
\end{align*}
Note that the ordinary \v{C}ech differential $\check\partial$ is a graded 
derivation with respect to these products:
\begin{lemma}\label{gradedderivation}
Let $\phi^{(k-1)1}\in \check{C}^{k-1}(\W,{\mathcal{T}}^1)$, $\phi^{(k-2)2}\in 
\check{C}^{k}(\W,\mathcal{T}^2)$, 
$A\in \check{C}^2(\W,\underline{\Z}^n)$ and $B\in \check{C}^3(\W,\underline{M_n^u(\Z)})$ 
be arbitrary cochains. Then
\begin{align*}
&\check\partial(\phi^{(k-1)1}\cup_1 A)=(\check\partial\phi^{(k-1)1})\cup_1 A\times\left[\phi^{(k-1)1}\cup_1 
(\check\partial A)\right]^{(-1)^{k-1}},\\
&\check\partial(\phi^{(k-2)2}\cup_1 A)=(\check\partial\phi^{(k-2)2})\cup_1 A\times\left[\phi^{(k-2)2}\cup_1 
(\check\partial A)\right]^{(-1)^{k-2}},\\
&\check\partial(\phi^{(k-2)2}\cup_2 B)=(\check\partial\phi^{(k-2)2})\cup_2 B\times\left[\phi^{(k-2)2}\cup_2 
(\check\partial B)\right]^{(-1)^{k-2}}.
\end{align*}
\end{lemma}
\begin{proof}
A simple computation.
\end{proof}

Now, as usual let $F_i$ denote the $i^{th}$ component of $F$, our fixed representative of the 
Euler vector of $\pi:X\to Z$. Applying Lemma \ref{commutecoboudary} with $A=F_i, B=F_j$ gives 
us a 3-cochain $C(F)\in \check{C}^3(\W,\underline{M_n^u(\Z)})$ defined by the formula:
\[C(F)_{\lambda_0\lambda_1\lambda_2\lambda_3}(z)_{ij}:=F_{\lambda_0\lambda_1\lambda_2}(z)_i
F_{\lambda_0\lambda_2\lambda_3}(z)_j-F_{\lambda_1\lambda_2\lambda_3}
(z)_iF_{\lambda_0\lambda_1\lambda_3}(z)_j.\]
Then, we define a map $D_F: C^k_F(\mathcal{W},\mathcal{S})\to C^{k+1}_F(\mathcal{W},\mathcal{S})$ by
\begin{align}
\label{dfdefn}&D_F(\phi^{k0}, \phi^{(k-1)1},\phi^{(k-2)2}):=\\
&(\check\partial\phi^{k0}\times(\phi^{(k-1)1}\cup_1 F)^{(-1)^{k+1}}\times(\phi^{(k-2)2}
\cup_2 C(F)))^{(-1)^{k+1}},\notag\\
&\quad\check\partial\phi^{(k-1)1}\times(\phi^{(k-2)2}\cup_1 F)^{(-1)^{k}},\check\partial\phi^{(k-2)2})\notag.
\end{align}
which one can see is a differential using Lemma \ref{gradedderivation}.
\begin{figure}
\centerline{\xymatrix{
&&&\\
&&&\\
\phi^{k0}\ar[uu]^{\check\partial}&&&&\\
&&\\
&&\phi^{(k-1)1}\ar[uu]^(0.4){\check\partial}\ar[uuuull]^{\cup_1 F}&&\\
&&&&\\
&&&&\phi^{(k-2)2}\ar[uu]^{\check\partial}\ar[uuuull]_{\cup_1 F}\ar[uuuuuullll]_(0.75){\cup_2 C(F)}
}}
\caption{The differential $D_{F}$}
\label{Mydifferential}
\end{figure}
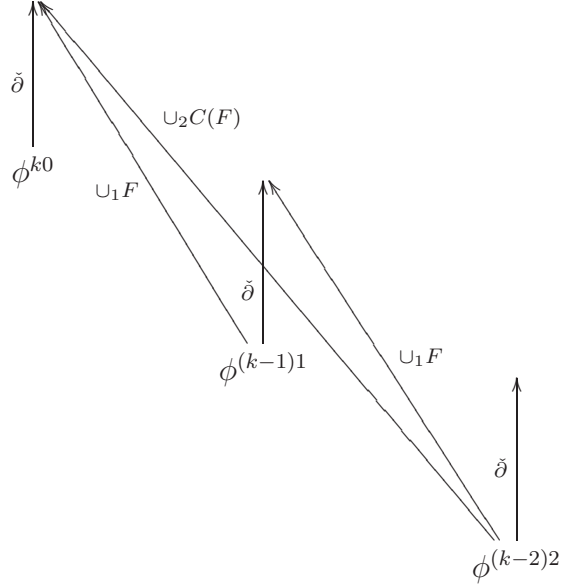

\begin{defn}
We define the \emph{$k^{th}$ dimensionally reduced \v{C}ech cohomology group of the covering 
$\W$ with coefficients in $\mathcal{S}$} to be the cohomology of $C^k_F(\mathcal{W},\mathcal{S})$  
under the differential $D_F$. This group is denoted $\mathbb{H}^k_F(\mathcal{W},\mathcal{S})$.
\end{defn}

We can also define a similar groups, $\mathbb{H}^k_F(\W,\underline{\Z})$ and 
$\mathbb{H}^k_F(\W,\mathcal{R})$, 
using integer  and real coefficients. We begin with integer coefficients. 
Cochains in $C^k_F(\W,\underline{\Z})$ 
are triples $(\phi^{k0},\phi^{(k-1)1},\phi^{(k-2)2})$, consisting of a \v{C}ech cochains 
$\phi^{k0}\in \check{C}^{k}(\W,\underline{\Z})$, $\phi^{(k-1)1}\in \check{C}^{k-1}(\W,\underline{\Z}^n)$ 
and $\phi^{(k-2)2}\in \check{C}^{k-2}(\W,\underline{M_n^u(\Z)})$. We define degree 0 and 1 cochains 
as before, by truncating the lower \v{C}ech cochains. To define the differential, let $m_l$ denote the $l^{th}$ 
component of $m\in\Z^n$.
Then we have maps
\begin{align*}
\cup_1 F&: \check{C}^{k-1}(\W,\underline{\Z}^n)\to \check{C}^{k+1}(\W,\underline{\Z}),\\
\cup_1 F&: \check{C}^{k-2}(\W,\underline{M_n^u(\Z)}) \to \check{C}^{k}(\W,\underline{\Z}^n), 
\quad \mbox{and}\\
\cup_2 C(F)&: \check{C}^{k-2}(\W,\underline{M_n^u(\Z)}) \to \check{C}^{k+1}(\W,\underline{\Z}),
\end{align*}
with their integer cohomology analogues:
\begin{align}
\label{cupformula2} & (\phi^{(k-1)1}\cup_1 F)_{\lambda_0\dots\lambda_{k+1}}(z)  :=
\sum_{l=1}^n\phi^{(k-1)1}_{\lambda_0\dots\lambda_{k-1}}(z)_l 
F_{\lambda_{k-1}\lambda_k\lambda_{k+1}}(z)_l\,,\\
& (\phi^{(k-2)2}\cup_1 F)_{\lambda_0\dots\lambda_{k}}(z)_l  :=  \notag \\ \qquad & 
\label{cupformula1} \sum_{1\leq i<j\leq n}\phi^{(k-2)2}_{\lambda_0
\dots\lambda_{k-2}}(z)_{ij}(F_{\lambda_{k-2}\lambda_{k-1}\lambda_{k}}(z)_i(e_l)_j-(e_l)_i
F_{\lambda_{k-2}\lambda_{k-1}\lambda_{k}}(z)_j)\,,  \\
& (\phi^{(k-2)2}\cup_2 C(F))_{\lambda_0\dots\lambda_{k+1}}(z)  :=\sum_{1\leq i<j\leq n}
\phi^{(k-2)2}_{\lambda_0\dots\lambda_{k-2}}(z)_{ij} C(F)_{\lambda_{k-2}\lambda_{k-1}\lambda_{k}\lambda_{k+1}}(z)_{ij},\notag
\end{align}
The formulas (\ref{cupformula1}) and (\ref{cupformula2}) should be compared with (\ref{wedgef1}) 
and (\ref{wedgef2}), respectively. Then the differential is
\begin{align*}
&D_F(\phi^{k0}, \phi^{(k-1)1},\phi^{(k-2)2}):=\\
&(\check\partial\phi^{k0}+(-1)^{k+1}\phi^{(k-1)1}\cup_1 F+(-1)^{k+1}\phi^{(k-2)2}\cup_2 C(F),\\
&\quad\check\partial\phi^{(k-1)1}+(-1)^{k}\phi^{(k-2)2}\cup_1 F,\check\partial\phi^{(k-2)2}).
\end{align*}

\begin{defn}
Let $\W$ be an open cover of a $C^\infty$ manifold $Z$, and fix a cocycle $F\in\check{Z}^2(\W,\underline{\Z}^n)$. 
We define the \emph{$k^{th}$ dimensionally reduced \v{C}ech cohomology group of the cover $\W$ with 
coefficients in $\Z$} to be the cohomology of $C^k_F(\mathcal{W},\underline{\Z})$  under the differential $D_F$. 
This group is denoted
$\mathbb{H}^k_F(\W,\underline{\Z})$.
\end{defn}

The definition of the real coefficient groups $\mathbb{H}^k_F(\W,\mathcal{R})$ are identical, except 
obviously that cochains take values in a different group. Thus, cochains in $C^k_F(\W,\mathcal{R})$ 
are triples $(\phi^{k0},\phi^{(k-1)1},\phi^{(k-2)2})$ consisting of a \v{C}ech cochains $\phi^{k0}\in 
\check{C}^{k}(\W,\mathcal{R})$, $\phi^{(k-1)1}\in \check{C}^{k-1}(\W,\mathcal{R}^n)$ and $\phi^{(k-2)2}\in 
\check{C}^{k-2}(\W,\mathcal{M}(\mathcal{R})),$ where $\mathcal{R}$ denotes the sheaf of germs of 
continuous $\R$-valued functions, and $\mathcal{M}(\mathcal{R})$ denotes the sheaf of germs of 
continuous $M_n^u(\R)$-valued functions.

\begin{prop}\label{cohomologylongexactsequence}
Let $\W$ be a good open cover of a $C^\infty$ manifold $Z$, and fix a cocycle $F\in \check{Z}^2(\W,\underline{\Z}^n)$. 
Then there is a long exact sequence of cohomology groups
\[\to\mathbb{H}^k_F(\W,\mathcal{R})\to\mathbb{H}^k_F(\W,\mathcal{S})\to 
\mathbb{H}^{k+1}_F(\W,\underline{\Z})\to \mathbb{H}^{k+1}_F(\W,\mathcal{R})\to\]
\end{prop}
\begin{proof}
Since the cover $\W$ is good we have an exact sequence of cochain complexes
\[0\to C^\bullet_F(\W,\underline{\Z})\to C^\bullet_F(\W,\mathcal{R})\to C^\bullet_F(\W,\mathcal{S})\to 0.\]
The exact sequence in cohomology then follows from standard homological algebraic techniques 
(see, for example, \cite[Lemma 4.31]{RaeWill98}).
\end{proof}

As one familiar with \v{C}ech cohomology might expect, most of the groups $\mathbb{H}^k_F(\W,\mathcal{R})$ 
are trivial. Given that we are fixing our covers and not allowing a locally finite 
refinement, to prove this fact we  need a definition and a lemma.

\begin{defn}
Let $\W=\{W_{\lambda}\}_{\lambda\in\mathcal{I}}$ be a (not necessarily locally finite) open cover of a 
paracompact space $Z$. Then we call a collection of functions $\{\rho_{\lambda}\}_{\lambda\in\mathcal{I}}$ 
with $\rho_\lambda\in C^\infty(Z,\R)$ a \emph{partition of unity subordinate to} $\W$ if and only if the collection 
satisfies the following conditions:
\begin{itemize}
\item[(1)] For all $\lambda\in\mathcal{I},\operatorname{range}\rho_\lambda\subset [0,1]$.
\item[(2)] For all $\lambda\in\mathcal{I},\operatorname{supp}\rho_\lambda\subset W_{\lambda}$.
\item[(3)] For all $z\in Z$, there are only finitely many $\rho_\lambda$ such that $\rho_\lambda(z)\neq 0$.
\item[(4)] For all $z\in Z$, $\sum_\lambda\rho_\lambda(z)=1$.
\end{itemize}
\end{defn}

\begin{lemma}\label{PartitionExistence}
Every open cover $\W=\{W_{\lambda}\}_{\lambda\in\mathcal{I}}$ of a $C^\infty$-manifold $Z$ has a partition 
of unity subordinate to it.
\end{lemma}
\begin{proof}
Since $Z$ is paracompact, there exists a locally finite refinement $\mathcal{V}=\{V_{\mu}\}_{\mu\in\mathcal{J}}$ 
of $\W=\{W_{\lambda}\}_{\lambda\in\mathcal{I}}$ that comes with a refinement map $\iota:\mathcal{J}\to\mathcal{I}$ 
satisfying $V_{\mu}\subset W_{\iota(\mu)}$. By standard results, there is a partition of unity 
$\{\varrho_{\mu}\}_{\mu\in\mathcal{J}}$ subordinate to $\mathcal{V}$. We then define a partition of 
unity $\{\rho_{\lambda}\}_{\lambda\in\mathcal{I}}$ subordinate to $\W$ with the formula
\[\rho_\lambda(z):=\sum_{\iota(\mu)=\lambda}\varrho_{\mu}(z).\]
\end{proof}

\begin{lemma}\label{Rcompute}
Let $\W$ be an open cover of of a $C^\infty$ manifold $Z$, and fix a cocycle 
$F\in \check{Z}^2(\W,\underline{\Z})$. Then we have group isomorphisms
\[\mathbb{H}^k_{F}(\W,\mathcal{R})\cong\begin{cases}
C(Z,\R) & k=0\\
C(Z,\R^n)& k=1\\
C(Z,M_n^u(\R))& k=2\\
0& k\geq 3.
\end{cases}\]
\end{lemma}
\begin{proof}
We only provide a proof of these facts for $k= 2$, since the other cases admit a similar proof. 
First recall that in real cohomology there is a contracting homotopy 
$h_k:\check{C}^k(\W,\mathcal{R})\to \check{C}^{k-1}(\W,\mathcal{R})$ defined with the assistance of a 
partition of unity $\{\rho_{\lambda}\}$ (that exists by Lemma \ref{PartitionExistence}) given by 
$h_k\phi_{\lambda_0\dots\lambda_{k-1}}=\sum_{\lambda}\rho_\lambda\phi_{\lambda\lambda_0\dots\lambda_{k-1}}$. 
Being a contracting homotopy, $\{h_k\}$ satisfies the identity $\operatorname{id}=\check\partial h_k+h_{k-1}\check\partial$. 
Now let us consider the case $k=2$. Consider a cocycle $(\phi^{20},\phi^{11},\phi^{02})$. This is 
cohomologically equivalent to
\begin{align*}
(\phi^{20},\phi^{11},\phi^{02}) +D_F(0, -h_1\phi^{11})& =(\phi^{20},\phi^{11},\phi^{02})+(-h_1\phi^{11}\cup_1
F, -\check\partial h_1\phi^{11},0)\\
=&(\phi^{20}-h_1\phi^{11}\cup_1F,\phi^{11}-\check\partial h_1\phi^{11},\phi^{02})\\
=&(\phi^{20}-h_1\phi^{11}\cup_1F,h_2\check\partial\phi^{11},\phi^{02})\\
=&(\phi^{20}-h_1\phi^{11}\cup_1F,h_2(\phi^{02}\cup_1 F),\phi^{02}).
\end{align*}
Let us define $\tilde{\phi}^{20}:=\phi^{20}-h_1\phi^{11}\cup_1F$, and, because $D_F$ is a differential, we can infer
\[\check\partial \tilde{\phi}^{20}=-h_2(\phi^{02}\cup_1 F)\cup_1 F-\phi^{02}\cup_2C(F).\]
Using the above facts we know $(\phi^{20},\phi^{11},\phi^{02})$ is cohomologically equivalent to
\begin{align*}
(\tilde{\phi}^{20},h_2(\phi^{02}\cup_1 F),\phi^{02})& +D_F(-h_2\tilde{\phi}^{02},0)
=(\tilde{\phi}^{20},h_2(\phi^{02}\cup_1 F),\phi^{02})+(-\check\partial h_2\tilde{\phi}^{20},0,0)\\
=&(h_3\check\partial \tilde{\phi}^{20},h_2(\phi^{02}\cup_1 F),\phi^{02})\\
=&(h_3(-h_2(\phi^{02}\cup_1 F)\cup_1 F-\phi^{02}\cup_2C(F)),h_2(\phi^{02}\cup_1 F),\phi^{02}).
\end{align*}
Thus, every cocycle $(\phi^{20},\phi^{11},\phi^{02})$ in $\mathbb{H}^2_F(\W,\mathcal{R})$ is equivalent to one of the form
\[(h_3(-h_2(\phi^{02}\cup_1 F)\cup_1 F-\phi^{02}\cup_2C(F)),h_2(\phi^{02}\cup_1 F),\phi^{02}),\]
and the isomorphism $\mathbb{H}^2_F(\W,\mathcal{R})\to C(Z,M_n^u(\R))$ is given by
\[[(\phi^{20},\phi^{11},\phi^{02})]\mapsto \phi^{02}.\]
\end{proof}

\begin{cor}\label{dimreducedlongexact}
Let $\W$ be a good open cover of a $C^\infty$ manifold $Z$, and fix a cocycle 
$F\in \check{Z}^2(\W,\underline{\Z}^n)$. Then we have exact sequences
\begin{align*}
0&\to C(Z,\Z)\to C(Z,\R)\to C(Z,\T)\to  \mathbb{H}^1_{F}(\W,\underline{\Z})\\
&\quad\to C(Z,\R^n)\to \mathbb{H}^1_{F}(\W,\mathcal{S})\to \mathbb{H}^2_{F}(\W,\underline{\Z})\\
&\quad\quad \to C(Z,M_n^u(\R))\to \mathbb{H}^2_{F}(\W,\mathcal{S})\to \mathbb{H}^3_{F}(\W,\underline{\Z})\to 0,\\
\end{align*}
and
\[0\to \mathbb{H}^k_{F}(\W,\mathcal{S})\to \mathbb{H}^{k+1}_{F}(\W,\underline{\Z})\to 0,\quad k\geq 3.\]
\end{cor}

\begin{cor}\label{firstcommute}
Let $\W$ be a good open cover of a $C^\infty$ manifold $Z$, and fix a cocycle 
$F\in\check{Z}^2(\W,\underline{\Z}^n)$. Let $\pi^*: \check{H}^{k}(\W,\mathcal{S})\to \mathbb{H}^{k}_{F}(\W,\mathcal{S})$
be the map
\[\pi^*:[\phi]\mapsto [\phi,1,1]\]
and define $\pi^*: \check{H}^{k}(\W,\underline{\Z})\to \mathbb{H}^{k}_{F}(\W,\underline{\Z})$ and 
$\pi^*: \check{H}^{k}(\W,\mathcal{R})\to \mathbb{H}^{k}_{F}(\W,\mathcal{R})$ similarly. Then there exists a 
commutative diagram with exact rows

\centerline{\xymatrix{
\ar[r]&\check{H}^k(\W,\mathcal{R})\ar[r]\ar[d]^{\pi^*}&\check{H}^k(\W,\mathcal{S})\ar[d]^{\pi^*}\ar[r]&
\check{H}^{k+1}(\W,\underline{\Z})\ar[d]^{\pi^*}\ar[r]&\check{H}^{k+1}(\W,\mathcal{R})\ar[d]^{\pi^*}\ar[r]&\\
\ar[r]&\mathbb{H}^k_F(\W,\mathcal{R})\ar[r]&\mathbb{H}^k_F(\W,\mathcal{S})\ar[r]&\ar[r] 
\mathbb{H}^{k+1}_F(\W,\underline{\Z})\ar[r]&\mathbb{H}^{k+1}_F(\W,\mathcal{R})\ar[r]&
}}
\end{cor}

\begin{proof}
The definitions are set up so that this works, so we only prove commutativity of the square
\centerline{\xymatrix{
\check{H}^k(\W,\mathcal{S})\ar[d]^{\pi^*}\ar[r]&\check{H}^{k+1}(\W,\underline{\Z})\ar[d]^{\pi^*}\\
\mathbb{H}^k_F(\W,\mathcal{S})\ar[r]& \mathbb{H}^{k+1}_F(\W,\underline{\Z}).
}}

\noindent Fix an element $[\phi]\in \check{H}^k(\W,\mathcal{S})$. Since $\W$ is good, there exists a 
cochain $\hat{\phi}\in \check{C}^k(\W,\mathcal{R})$ such that $[\hat{\phi}]_{\R/\Z}=\phi$. Then the 
image of $[\phi]$ under $\check{H}^k(\W,\mathcal{S})\to \check{H}^{k+1}(\W,\underline{\Z})$ is 
$[\check\partial\hat{\phi}]$. The image of this under $\pi^*:H^{k+1}(\W,\underline{\Z})\to 
\mathbb{H}^{k+1}_F(\W,\underline{\Z})$ is $[\check\partial \hat{\phi},0,0]$.

Going anticlockwise on the other hand, the image of $[\phi]$ in $\mathbb{H}^k_F(\W,\mathcal{S})$ is $[\phi,1,1]$. 
This has image $[\check\partial\hat{\phi},0,0]$ in $\mathbb{H}^{k+1}_F(\W,\underline{\Z})$.
\end{proof}

\section{Curvature Groups}

There is a group related to $\mathbb{H}^k_{F}(\W,\mathcal{S})$ that is important to us, 
because later it will be the target of a Gysin ``integration over the fibres" map. Moreover, this 
group provides a curvature class for a certain class of noncommutative torus bundles. 
We define a cochain complex $\overline{C}^k_F(\mathcal{W},\mathcal{S})$, where for $k\geq 1$ 
an element is a pair $(\varphi^{k1},\varphi^{(k-1)2})$ consisting of \v{C}ech cochains $\phi^{k1}\in 
\check{C}^{k}(\W,\hat{\mathcal{N}})$ and $\phi^{(k-1)2}\in \check{C}^{k-1}(\W,\mathcal{M})$. 
A cochain in $\overline{C}^0_F(\mathcal{W},\mathcal{S})$ is given by $(\varphi^{01})$, 
for $\phi^{01}\in \check{C}^{0}(\W,\hat{\mathcal{N}})$. This complex has a differential $\overline{D}_F$ given by
\begin{equation}
\label{overdfdefn}\overline{D}_F(\phi^{k1},\phi^{(k-1)2}):=(\check\partial\phi^{(k-1)1}
\times(\phi^{(k-2)2}\cup_1 F)^{(-1)^{k}},\check\partial\phi^{(k-2)2}).
\end{equation}
\begin{defn}
We define $\overline{\mathbb{H}}^k_F(\mathcal{W},\mathcal{S})$ to be the cohomology of 
$\overline{C}^k_F(\mathcal{W},\mathcal{S})$  under the differential $\overline{D}_F$.
\end{defn}
One can see this group is obtained by removing the first entry from $\overline{C}^{k+1}_{F}(\W,\mathcal{S})$. 
Obviously, one can define the same groups with integer and real coefficients, denoted 
$\overline{\mathbb{H}}^k_F(\mathcal{W},\underline{\Z})$ and $\overline{\mathbb{H}}^k_F(\mathcal{W},\mathcal{R})$ 
respectively. Doing so gives us analogues of Proposition \ref{cohomologylongexactsequence} and 
Lemma \ref{Rcompute}:

\begin{prop}\label{exponentiallong2}
Let $\W$ be a good open cover of a $C^\infty$ manifold $Z$, and fix a cocycle $F\in \check{Z}^2(\W,\underline{\Z}^n)$. 
Then there is a long exact sequence of cohomology groups
\[\ldots\to\overline{\mathbb{H}}^k_F(\W,\mathcal{R})\to\overline{\mathbb{H}}^k_F(\W,\mathcal{S})\to 
\overline{\mathbb{H}}^{k+1}_F(\W,\underline{\Z})\to \overline{\mathbb{H}}^{k+1}_F(\W,\mathcal{R})\to\ldots\]
\end{prop}

\begin{lemma}
Let $\W$ be an open cover of of a $C^\infty$ manifold $Z$, and fix a cocycle 
$F\in \check{Z}^2(\W,\underline{\Z})$. Then we have group isomorphisms
\[\overline{\mathbb{H}}^k_{F}(\W,\mathcal{R})\cong\begin{cases}
C(Z,\R^n)& k=0\\
C(Z,M_n^u(\R))& k=1\\
0& k\geq 2.
\end{cases}\]
\end{lemma}

\begin{lemma}\label{middlecommute}
Let $\W$ be a good open cover of a $C^\infty$ manifold $Z$, and fix a cocycle 
$F\in\check{Z}^2(\W,\underline{\Z}^n)$. Let $\pi_*: \mathbb{H}^{k}_{F}(\W,\mathcal{S})\to 
\overline{\mathbb{H}}^{k-1}_{F}(\W,\mathcal{S})$
be the map
\[\pi_*:[\phi^{k0},\phi^{(k-1)1},\phi^{(k-2)2}]\mapsto [\phi^{(k-1)1},\phi^{(k-2)2}],\]
and define $\pi_*: \mathbb{H}^{k}_{F}(\W,\underline{\Z})\to\overline{\mathbb{H}}^{k-1}_{F}(\W,\underline{\Z})$ and 
$\pi_*: \mathbb{H}^{k}_{F}(\W,\mathcal{R})\to\overline{\mathbb{H}}^{k-1}_{F}(\W,\mathcal{R})$ similarly. 
Then there is a commutative diagram with exact rows

\centerline{\xymatrix{
\ar[r]&\mathbb{H}^{k-1}_{F}(\W,\mathcal{R})\ar[r]\ar[d]^{\pi_*}&\mathbb{H}^{k-1}_{F}(\W,\mathcal{S})\ar[r]\ar[d]^{\pi_*}& 
\mathbb{H}^{k}_{F}(\W,\underline{\Z})\ar[d]^{\pi_*}\ar[r]&\mathbb{H}^{k}_{F}(\W,\mathcal{R})\ar[d]^{\pi_*}\ar[r]&\\
\ar[r]&\overline{\mathbb{H}}^{k-2}_{F}(\W,\mathcal{R})\ar[r]&\overline{\mathbb{H}}^{k-2}_{F}(\W,\mathcal{S})\ar[r]& 
\overline{\mathbb{H}}^{k-1}_{F}(\W,\underline{\Z})\ar[r]&\overline{\mathbb{H}}^{k-1}_{F}(\W,\mathcal{R})\ar[r]&
}}
\end{lemma}
\begin{proof}
As with Corollary \ref{firstcommute}, the definitions are set up so that this works. We prove only 
commutativity of the the square

\centerline{\xymatrix{
\mathbb{H}^{k-1}_{F}(\W,\mathcal{S})\ar[r]\ar[d]^{\pi_*}& \mathbb{H}^{k}_{F}(\W,\underline{\Z})\ar[d]^{\pi_*}\\
\overline{\mathbb{H}}^{k-2}_{F}(\W,\mathcal{S})\ar[r]& \overline{\mathbb{H}}^{k-1}_{F}(\W,\underline{\Z}).
}}

\noindent Fix a class $[\phi^{(k-1)0},\phi^{(k-2)1},\phi^{(k-3)2}]\in \mathbb{H}^{k-1}_{F}(\W,\mathcal{S})$. 
We first proceed clockwise around the diagram. Since $\W$ is good there exists a 
cochain $[\hat{\phi}^{(k-1)0},\hat{\phi}^{(k-2)1},\hat{\phi}^{(k-3)2}]$ $\in C^{k-1}_{F}(\W,\mathcal{R})$ such that
\begin{align*}
&[\hat{\phi}^{(k-1)0}]_{\R/\Z}=\phi^{(k-1)0},\\
&[\hat{\phi}^{(k-2)1}(\cdot)_l]_{\R/\Z}=\phi^{(k-2)1}(e_l,\cdot), \quad 1\leq l\leq n,\\
&[\hat{\phi}^{(k-3)2}]_{M_n^u(\R/\Z)}=\phi^{(k-3)2}.
\end{align*}
Then the image of $[\phi^{(k-1)0},\phi^{(k-2)1},\phi^{(k-3)2}]$ in $\mathbb{H}^{k}_{F}(\W,\underline{\Z})$ is given by
\begin{align*}
[&\check\partial \hat{\phi}^{(k-1)0} +(-1)^{k+1}\hat{\phi}^{(k-2)1}\cup_1 F+(-1)^{k+1}\hat{\phi}^{(k-3)2}\cup_2 C(F),\\
&\quad \check\partial\hat{\phi}^{(k-2)1}+(-1)^{k}\hat{\phi}^{(k-3)2}\cup_1 F, \check\partial \hat{\phi}^{(k-3)2}].
\end{align*}
This class will have image under $\pi_*: \mathbb{H}^{k}_{F}(\W,\underline{\Z})\to 
\overline{\mathbb{H}}^{k-1}_{F}(\W,\underline{\Z})$ the class
\begin{align*}
[\check\partial\hat{\phi}^{(k-2)1}+(-1)^{k}\hat{\phi}^{(k-3)2}\cup_1 F, \check\partial \hat{\phi}^{(k-3)2}].
\end{align*}
This is exactly what we get if we go anticlockwise around the diagram.
\end{proof}

\begin{lemma}\label{lastcommute}
Let $\W$ be a good open cover of a $C^\infty$ manifold $Z$, and fix a cocycle 
$F\in\check{Z}^2(\W,\underline{\Z}^n)$. Let $\cup F: \overline{\mathbb{H}}^{k-1}_{F}(\W,\mathcal{S})\to \check{H}^{k+1}(\W,
\mathcal{S})$ be the map
\[\cup F:[\phi^{(k-1)1},\phi^{(k-2)2}]\mapsto [(\phi^{(k-1)1}\cup_1 F)^{(-1)^{k+2}}\times(\phi^{(k-2)2}\cup_2 C(F))^{(-1)^{k+2}}],\]
and define $\cup F: \overline{\mathbb{H}}^{k-1}_{F}(\W,\underline{\Z})\to \check{H}^{k+1}(\W,\underline{\Z})$ 
and $\cup F: \overline{\mathbb{H}}^{k-1}_{F}(\W,\mathcal{R})\to \check{H}^{k+1}(\W,\mathcal{R})$ similarly. 
Then there is a commutative diagram

\centerline{\xymatrix{
\ar[r]&\overline{\mathbb{H}}^{k-1}_{F}(\W,\mathcal{R})\ar[r]\ar[d]^{\cup F}&
\overline{\mathbb{H}}^{k-1}_{F}(\W,\mathcal{S})\ar[r]\ar[d]^{\cup F}& 
\overline{\mathbb{H}}^{k}_{F}(\W,\underline{\Z})\ar[r]\ar[d]^{\cup F}& 
\overline{\mathbb{H}}^{k}_{F}(\W,\mathcal{R})\ar[r]\ar[d]^{\cup F}&\\
\ar[r]&H^{k+1}(\W,\mathcal{R})\ar[r]&H^{k+1}(\W,\mathcal{S})\ar[r]&H^{k+2}(\W,\underline{\Z})
\ar[r]&H^{k+2}(\W,\mathcal{R})\ar[r]&\\
}}
\end{lemma}

\begin{proof}
The formula for $\cup F: \overline{\mathbb{H}}^{k-1}_{F}(\W,\mathcal{S})\to \check{H}^{k+1}(\W,\mathcal{S})$ 
should be compared with the definition of the differential $D_F$ from Equation (\ref{dfdefn}). That it is a 
well-defined map of cohomology groups follows from an easy  calculation.
Commutativity then follows from the definitions and the fact that $\W$ is good, just as for the proofs of 
Corollary \ref{firstcommute} and Lemma \ref{middlecommute}.
\end{proof}

\section{The Gysin Sequence}

\begin{theorem}\label{MyGysin}
Let $\W$ be a good open cover of a $C^\infty$ manifold $Z$, and fix a cocycle 
$F\in \check{Z}^2(\W,\underline{\Z}^n)$. Then there is a commuting diagram with exact columns and rows:

\centerline{\xymatrix{
\ar[r]&\check{H}^k(\W,\mathcal{R})\ar[r]^{\pi^*}\ar[d]& \mathbb{H}^k_F(\W,\mathcal{R})\ar[r]^{\pi_*}\ar[d]& 
\overline{\mathbb{H}}^{k-1}_{F}(\W,\mathcal{R})\ar[d]\ar[r]^{\cup F}&\check{H}^{k+1}(\W,\mathcal{R})\ar[d]\ar[r]&\\
\ar[r]&\check{H}^k(\W,\mathcal{S})\ar[r]^{\pi^*}\ar[d]& \mathbb{H}^k_F(\W,\mathcal{S})\ar[r]^{\pi_*}\ar[d]& 
\overline{\mathbb{H}}^{k-1}_{F}(\W,\mathcal{S})\ar[d]\ar[r]^{\cup F}&\check{H}^{k+1}(\W,\mathcal{S})\ar[d]\ar[r]&\\
\ar[r]&\check{H}^{k+1}(\W,\underline{\Z})\ar[r]^{\pi^*}& \mathbb{H}^{k+1}_F(\W,\underline{\Z})\ar[r]^{\pi_*}& 
\overline{\mathbb{H}}^{k}_{F}(\W,\underline{\Z})\ar[r]^{\cup F} & \check{H}^{k+2}(\W,\underline{\Z})\ar[r]&\\
}}
\end{theorem}

\begin{proof}
Commutativity is the content of Corollary \ref{firstcommute}, Lemma \ref{middlecommute} and Lemma 
\ref{lastcommute}. Exactness in the vertical directions comes from Proposition \ref{cohomologylongexactsequence} 
and Proposition \ref{exponentiallong2} (as well as the ordinary \v{C}ech cohomology ``exponential" long exact sequence). 
Exactness in the horizontal direction is assured by the definitions; the most difficult part is exactness in the 
horizontal direction at $\overline{\mathbb{H}}^{k-1}_{F}(\W,\mathcal{S})$, which we now prove. First, we 
show $\cup F\circ\pi_*=1$. Suppose $[\phi^{k0},\phi^{(k-1)1},\phi^{(k-2)2}]\in \mathbb{H}^k_F(\W,\mathcal{S})$. 
By definition we have
\[(\pi_*[\phi^{k0},\phi^{(k-1)1},\phi^{(k-2)2}])\cup F:=[(\phi^{(k-1)1}\cup_1 F)^{(-1)^{k+2}}\times(\phi^{(k-2)2}
\cup_2 C(F))^{(-1)^{k+2}}].\]
However, since $(\phi^{k0},\phi^{(k-1)1},\phi^{(k-2)2})$ is closed under $D_F$, we can infer
\[\check\partial\phi^{k0}=(\phi^{(k-1)1}\cup_1 F)^{(-1)^{k+2}}\times(\phi^{(k-2)2}\cup_2 C(F))^{(-1)^{k+2}},\]
so that
\[(\pi_*[\phi^{k0},\phi^{(k-1)1},\phi^{(k-2)2}])\cup F=[\check\partial\phi^{k0}]=1\in \check{H}^{k+1}(\W,\mathcal{S}),\]
and therefore $\cup F\circ\pi_*=1$.

On the other hand, if $[\phi^{(k-1)1},\phi^{(k-2)2}]\in \overline{\mathbb{H}}^{k-1}_{F}(\W,\mathcal{S})$ 
is such that \[[\phi^{(k-1)1},\phi^{(k-2)2}]\cup F=1\in \check{H}^{k+1}(\W,\mathcal{S}),\]
then there exists a \v{C}ech $k$-cochain $\phi^{k}\in \check{C}^{k}(\W,\mathcal{S})$ such that
\[\check\partial\phi^{k}=(\phi^{(k-1)1}\cup_1 F)^{(-1)^{k+2}}\times(\phi^{(k-2)2}\cup_2 C(F))^{(-1)^{k+2}}.\]
Then $[\phi^{k},\phi^{(k-1)1},\phi^{(k-2)2}]$ is an element of $\mathbb{H}^k_F(\W,\mathcal{S})$ that 
satisfies \[\pi_*[\phi^{k},\phi^{(k-1)1},\phi^{(k-2)2}]=[\phi^{(k-1)1},\phi^{(k-2)2}].\]
\end{proof}

\section[Correspondence with the work of Bouwknegt et al.]{Correspondence 
with the work of Bouwknegt, Hannabuss and Mathai}\label{BHMwork}

We show here that the Gysin sequence given in Theorem \ref{MyGysin} corresponds with the 
Gysin sequence from Theorem \ref{BHMGysin}, in the sense of the following theorem.
\begin{theorem}\label{AgreementwithBHM}
Let $Z$ be a $C^\infty$ manifold, $\W$ a good open cover of $Z$. Fix a class $F\in \check{Z}^2(\W,\underline{\Z}^n)$, 
and denote the image of $F$ under the \v{C}ech-de Rham isomorphism by $F_2\in\Omega^2(Z,\t)$. 
Then there is a commutative diagram with exact rows:

\centerline{\xymatrix{
\ar[r]&\check{H}^{k}(\W,\underline{\Z})\ar[r]\ar[d]& \mathbb{H}^{k}_F(\W,\underline{\Z})\ar[r]\ar[d]& 
\overline{\mathbb{H}}^{k-1}_{F}(\W,\underline{\Z})\ar[r]\ar[d] & \check{H}^{k+1}(\W,\underline{\Z})\ar[r]\ar[d]&\\
\ar[r]&H^{k}_{dR}(Z)\ar[r]& H^{k,(0,2)}_{F_2}(Z,\t^*)\ar[r]& H^{k,(1,2)}_{F_2}(Z,\t^*)\ar[r] & H^{k+1}_{dR}(Z)\ar[r]&\\
}}

\end{theorem}

The difficult part of this theorem is the construction of the maps
\[\mathbb{H}^{k}_F(\W,\underline{\Z})\to H^{k,(0,2)}_{F_2}(Z,\t^*).\]
Recall an element in $\mathbb{H}^k_{F}(\W,\underline{\Z})$ is the class of a triple 
$(\phi^{k0},\phi^{(k-1)1},\phi^{(k-2)2})\in 
C_F^k(\W,\Z)$. The image of $[\phi^{k0},\phi^{(k-1)1},\phi^{(k-2)2}]$ 
will be a class $[H_{k0},H_{(k-1)1},H_{(k-2)2}]\in H^{k,(0,2)}_{F_2}(Z,\t^*)$. We shall define the 
components $H_{k0},H_{(k-1)1}$ and $H_{(k-2)2}$  in terms of $\phi^{k0},\phi^{(k-1)1}$ and $\phi^{(k-2)2}$ 
over the next page. By Lemma \ref{PartitionExistence} we may assume there is a partition of 
unity $\{\rho_{\lambda_0}\}$ subordinate to $\W$. 
Using the ``collating formula" for the \v{C}ech-de Rham isomorphism from \cite[Prop 9.5]{BotTu82}  
(see also \cite[Prop 1.4.17]{Bry93} ), we have images of $F$ and $\phi^{(k-2)2}$ in $\Omega^2(Z,\t)$ 
and $\Omega^{k-2}(Z,\wedge^2\t^*)$, respectively, given by
\begin{align}
\label{F2rep}(F_2)_l|_{W_{\lambda_0}}:=&\sum_{\lambda_1,\lambda_2}F_{\lambda_0\lambda_1\lambda_2}(\cdot)_l
d\rho_{\lambda_1}d\rho_{\lambda_2},\\
(H_{(k-2)2})_{ij}|_{W_{\lambda_0}}:=&\sum_{\lambda_1,\dots, \lambda_{k-2}}
\phi^{(k-2)2}_{\lambda_0\dots\lambda_{k-2}}(\cdot)_{ij}d\rho_{\lambda_1}\dots 
d\rho_{\lambda_{k-2}}\,,\qquad 1\leq i<j\leq n\notag\,,
\end{align}
where we have suppressed the wedge $\wedge$. To define $H_{(k-1)1}$, we have to take into account that 
$\phi^{(k-1)1}$ is not closed under $\check\partial$. Therefore, using the same collating formula, we 
have a global differential form given by
\begin{align*}
(H_{(k-1)1})_{l}|_{W_{\lambda_0}}:=&\sum_{\lambda_1,\dots, \lambda_{k-1}}\phi^{(k-1)1}_{\lambda_0
\dots\lambda_{k-1}}(\cdot)_{l}d\rho_{\lambda_1}\dots d\rho_{\lambda_{k-1}}\\
&\quad +\sum_{\lambda_{1},\dots,\lambda_k}\check\partial\phi^{(k-1)1}_{\lambda_{0}
\dots\lambda_k}(\cdot)_{l}\rho_{\lambda_1}d\rho_{\lambda_2}\dots d\rho_{\lambda_{k}}\,,\qquad 1\leq l\leq n.
\end{align*}
Defining $H_{k0}$ is a little trickier because $\phi^{k0}$ get a contribution from  $\phi^{(k-2)2}$, via the differential $D_F$.
This does not occur in the dimensionally reduced cohomology of 
\cite{BouHanMat05} (cf. Figs \ref{Mydifferential} and \ref{BHMdifferential}). Fortunately, 
it turns out this contribution is an exact form. Define a global differential form $D(\phi^{(k-2)2}, C(F))$ by the formulas:
\begin{align*}
&D(\phi^{(k-2)2}, C(F))|_{W_{\lambda_0}}\\
=&\begin{cases}
\sum_{\stackrel{1\leq i<j\leq n}{\lambda_1,\dots, \lambda_4}}\phi^{02}(\cdot)_{ij}C(F)_{\lambda_1\dots  
\lambda_4}(\cdot)_{ij}\rho_{\lambda_1}\rho_{\lambda_3}d\rho_{\lambda_2} d\rho_{\lambda_4} & k=2\\
\sum_{\stackrel{1\leq i<j\leq n}{\lambda_1,\dots, \lambda_4}}\phi^{12}_{\lambda_2\lambda_1}(\cdot)_{ij}
C(F)_{\lambda_1\dots \lambda_4}(\cdot)_{ij}\rho_{\lambda_1}d\rho_{\lambda_2}d\rho_{\lambda_3} d\rho_{\lambda_4} & k=3\\
\sum_{\stackrel{1\leq i<j\leq n}{\lambda_1,\dots,\lambda_{k+1}}}\phi_{\lambda_{k-1}\lambda_1
\dots\lambda_{k-2}}^{(k-2)2}(\cdot)_{ij}C(F)_{\lambda_{k-2}\dots\lambda_{k+1}}(\cdot)_{ij}
\rho_{\lambda_1}d\rho_{\lambda_2}\dots d\rho_{\lambda_{k+1}}& k>3
\end{cases}
\end{align*}
Then we define
\begin{align*}
(H_{k0})|_{W_{\lambda_0}}:=&\sum_{\lambda_1,\dots, \lambda_{k}}\phi^{k0}_{\lambda_0\dots
\lambda_{k}}(\cdot)_{l}d\rho_{\lambda_1}\dots d\rho_{\lambda_{k}}\\
&\quad +\sum_{\lambda_{1},\dots,\lambda_{k+1}}\check\partial\phi^{k0}_{\lambda_{0}
\dots\lambda_{k+1}}(\cdot)_{l}\rho_{\lambda_1}d\rho_{\lambda_2}\dots d\rho_{\lambda_{k+1}}\\
&\quad +(-1)^{k+1}D(\phi^{(k-2)2}, C(F))|_{W_{\lambda_0}}.
\end{align*}

\begin{remark}
Examining the definition of $\mathbb{H}^k_F(\W,\underline{\Z})$, it is clear that we could also 
define a collection of cohomology groups ``$\mathbb{H}^k_F(\W,\underline{\R})$" using the c
onstant sheaf with values in $\R$. Then the map \[(\phi^{k0},\phi^{(k-1)1},\phi^{(k-2)2})\mapsto(H_{k0},H_{(k-1)1},H_{(k-2)2})\]
given above factors through the map similarly defined from $\mathbb{H}^k_F(\W,\underline{\R})$ to \\
$H^{k,(0,2)}_{F_2}(Z,\t^*)$. We believe this latter map gives an isomorphism 
$\mathbb{H}^2_F(\W,\underline{\R})\cong H^{2,(0,2)}_{F_2}(Z,\t^*)$, but, in absence of an inverse 
for the \v{C}ech-de Rham collating formula, we do not have a proof. Thus, we have chosen to avoid the issue.
\end{remark}

\begin{lemma}\label{cochainmap}
Let $Z$ be $C^\infty$ manifold, $\W$ a good open cover of $Z$, and $F\in \check{Z}^2(\W,\underline{\Z}^n)$. 
Let the image of $F$ in $\Omega^2(Z,\t)$ under the \v{C}ech-de Rham isomorphism be $F_2$. Then the map
\[C_{F}^k(\W,\underline{\Z})\to C_{F_2}^{k,(0,2)}(Z,\t^*)\]
defined above by
\[(\phi^{k0},\phi^{(k-1)1},\phi^{(k-2)2})\mapsto(H_{k0},H_{(k-1)1},H_{(k-2)2})\]
is a map of cochain complexes.
\end{lemma}

\begin{proof}
The proof of this lemma is a long, unenlightening calculation, and has thus been relegated to the appendix.
\end{proof}

\begin{proof}[Proof of Thm \ref{AgreementwithBHM}]
Exactness for both sequences has already been done inTheorem \ref{MyGysin} for the top row and 
Theorem \ref{BHMGysin} for the bottom. The downward arrow $\mathbb{H}^k_F(\W,\underline{\Z})\to 
H^{k,(0,2)}_{F_2}(Z,\t^*)$ is defined by
\[[\phi^{k0},\phi^{(k-1)1},\phi^{(k-2)2}]\mapsto [H_{k0},H_{(k-1)1},H_{(k-2)2}],\]
and the downward arrow $\overline{\mathbb{H}}^{k-1}_F(\W,\underline{\Z})\to H^{k,(1,2)}_{F_2}(Z,\t^*)$ is defined by
\[[\phi^{(k-1)1},\phi^{(k-2)2}]\mapsto [H_{(k-1)1},H_{(k-2)2}].\]
The remaining downward arrow is just the \v{C}ech-de Rham Collating formula
\[\phi\mapsto \sum_{\lambda_1\dots\lambda_k}\phi_{\lambda_0\dots\lambda_k}d\rho_{\lambda_1}\dots d\rho_{\lambda_k}.\]
For commutativity, the only square that is not immediate from the definitions is

\centerline{\xymatrix{
\overline{\mathbb{H}}^{k-1}_F(\W,\underline{\Z})\ar[r]\ar[d]& H^{k+1}(\W,\underline{\Z})\ar[d]\\
H^{k,(1,2)}_{F_2}(Z,\t^*)\ar[r]& H_{dR}^{k+1}(Z),
}}

\noindent which itself is quite nontrivial. Going anticlockwise around the diagram, we have that the image 
of a class $[\phi^{(k-1)1},\phi^{(k-2)2}]$ in $H_{dR}^{k+1}(Z)$ is the class of the differential $k+1$-form
\begin{align*}\label{antiHwedgeF}
(-1)^{k}H_{(k-1)1}\wedge F_2 \,,
\end{align*}
where we recall $F_2$ is the image of $F$ under the collating formula.

On the other hand, the image of $[\phi^{(k-1)1},\phi^{(k-2)2}]$ in $H^{k+1}(\W,\underline{\Z})$ is
\[[(-1)^{k+2}\phi^{(k-1)1}\cup_1F+(-1)^{k+2}\phi^{(k-2)2}\cup_2C(F)].\]
Thus, we need to check there is a differential $k$-form $\omega$ such that the image of
\[(-1)^{k+2}\phi^{(k-1)1}\cup_1F+(-1)^{k+2}\phi^{(k-2)2}\cup_2C(F)\,,\]
in de Rham cohomology, is equal to $(-1)^{k}H_{(k-1)1}\wedge F_2-d\omega$. However, 
in (\ref{cancelledterm}) to (\ref{difffirst}) in the appendix, it is shown that we can take $\omega$ to equal
\[\omega=(-1)^{k+1}D(\phi^{(k-2)2},C(F)) \,,\]
for this purpose.
\end{proof}

\section{Mathai-Rosenberg T-Duality}

We return to the applications of our work to T-duality of principal torus bundles by first recalling the 
definition of T-duality according to Mathai and Rosenberg \cite{MatRos05, MatRos06}. 
Let $\check{H}^3(X,\Z)|_{\pi^{0,3}=0}$ denote the kernel of the Serre spectral sequence 
projection $\pi^{0,3}:\check{H}^3(X,\underline{\Z})\to C(Z,\check{H}^3(\T^n,\underline{\Z}))$. 
Suppose that $\delta\in\check{H}^3(X,\underline{\Z})$, and let $CT(X,\delta)$ denote the unique 
(up to $C_0(X)$-linear isomorphism) stable continuous trace algebra with Dixmier-Douady class $\delta$. 
Then, if $(c,\delta)$ is a T-duality pair such that $\delta\in \check{H}^3(X,\underline{\Z})|_{\pi^{0,3}=0}$, 
Theorem 2.2 of \cite{MatRos06} implies there is an a action of $\R^n$ on $CT(X,\delta)$ such that the 
induced action of $\R^n$ on $X$ covers the $\T^n$-bundle action. The \emph{Mathai-Rosenberg T-dual} 
is by definition the $C^*$-algebra $CT(X,\delta)\rtimes_\alpha\R^n$.

The $C^*$-algebra $CT(X,\delta)\rtimes_\alpha\R^n$ can be interpreted as the algebra of sections of a 
noncommutative torus bundle over $Z$, such that, if $f\in C(Z,M^u_n(\T))$ is the Mackey 
obstruction map of $\alpha$, then the fibre above $z\in\Z$ is the stabilised noncommutative torus 
$A_{f(z)}\otimes \K$ \cite[Proof of Thm 3.1]{MatRos06}. Moreover, if $\check{H}^3(X,
\underline{\Z})|_{\pi^{1,2}=0}$ denotes the kernel of the Serre spectral sequence projection $\pi^{1,2}:
\check{H}^3(X,\underline{\Z})|_{\pi^{0,3}=0}\to \check{H}^1(Z,\check{H}^2(\T^n,\underline{\Z}))$, 
then if $\delta\in \check{H}^3(X,\underline{\Z})|_{\pi^{1,2}=0}$ the action $\alpha$ can be chosen 
to have trivial Mackey obstruction. This implies that the $C^*$-algebra $CT(X,\delta)\rtimes_\alpha\R^n$ 
has spectrum $\hat{X}$, such that the dual $\R^n$-action induces a classical principal torus bundle 
$\hat{\pi}:\hat{X}\to Z$.

\section{T-Duality for Principal Torus Bundles with H-Flux via the Integer Gysin Sequence}

We continue the T-duality discussion from the introduction. There we finished by describing how the 
Gysin sequence of Theorem \ref{BHMGysin} allows one to compute the T-dual curvature of a T-duality
pair $(c,\delta)\in\check{H}^2(Z,\underline{\Z}^n)\oplus \check{H}^3(X,\underline{\Z})$ as the image of $\delta$ under the 
composition of the dimensional reduction isomorphism and the Gysin sequence integration over the fibres map:
\[\delta\mapsto[H_3,H_2,H_1]\to [H_2,H_1]\in H_{F_2}^{3,(1,2)}(Z,\mathfrak{t}^*).\]
Since Theorem \ref{AgreementwithBHM} shows the Gysin sequence of Theorem \ref{MyGysin} agrees 
with the one from Theorem \ref{BHMGysin}, we seek to provide an analogue of the dimensional 
reduction isomorphism. Indeed, this theorem is proved in a companion paper:

\begin{theorem}[\cite{BouCarRat2}]\label{BouCarRat2Theorem}
Let $\pi:X\to Z$ be a principal $\T^n$-bundle over a Riemannian manifold $Z$, and denote by 
$\check{H}^3(X,\underline{\Z})|_{\pi^{0,3}=0}$ the kernel of the Serre spectral sequence projection 
$\pi^{0,3}:\check{H}^3(X,\underline{\Z})\to \check{H}^3(\T^n,\underline{\Z})$. Then there exists an 
open cover $\U$ of $X$ such that $\pi:X\to Z$ has Euler vector $[F]\in \check{H}^2(\pi(\U),\underline{\Z}^n)$, 
and such that every $c\in \check{H}^3(X,\underline{\Z})|_{\pi^{0,3}=0}$ has a representative $[H]\in 
\check{H}^3(\U,\underline{\Z})$. Moreover, there is an isomorphism
\[\check{H}^3(X,\underline{\Z})|_{\pi^{0,3}=0}\cong \mathbb{H}^3_F(\pi(\U),\underline{\Z}).\]
\end{theorem}

\begin{defn}
Let $(c,\delta)\in \check{H}^2(Z,\underline{\Z}^n)\oplus \check{H}^3(X,\underline{\Z})$ 
be a T-duality pair, and let  $\pi:X\to Z$ be a  principal 
$\T^n$-bundle classified by $c \in \check{H}^2(Z,\Z^n)$. Fix an open cover $\U$ and representatives 
$[F]\in\check{H}^2(\pi(\U),\underline{\Z}^n)$ and $[H]\in\check{H}^3(\U,\underline{\Z})$ of $c$ and $\delta$ respectively, and 
let $[\phi^{30},\phi^{21},\phi^{12}]$ be the image of $[H]$ in $\mathbb{H}^3_F(\pi(\U),\underline{\Z})$ 
under the isomorphism from Theorem \ref{BouCarRat2Theorem}. Then we define the \emph{T-dual 
Euler vector} of $(c,\delta)$ to be the class $\pi_*[\phi^{30},\phi^{21},\phi^{12}]=[\phi^{21},\phi^{12}]\in 
\overline{\mathbb{H}}^2_F(\pi(\U),\underline{\Z})$.
\end{defn}

Let us describe how this definition fits into existing work. First, we need a lemma from our companion paper 
\cite{BouCarRat2}:

\begin{lemma}[\cite{BouCarRat2}]\label{fibreagree}
Fix a T-duality pair $(c,\delta)$, an open cover $\U$ of $X$, and classes 
$[F]\in\check{H}^2(\pi(\U),\underline{\Z}^n)$ and $[H]\in\check{H}^3(\U,\underline{\Z})$ as in 
Theorem \ref{BouCarRat2Theorem}, and suppose that $[H]$ maps to 
$[H_3,H_2,H_1]\in\mathbb{H}^3(\pi(\U),\underline{\Z})$. Let $\alpha$ be 
an action of $\R^n$ on $CT(X,[H])$ with Mackey obstruction $f\in C(Z,\T^{n\choose 2})$. 
Then the homotopy class of $f$ in $\check{H}^1(Z,\underline{\Z}^{n\choose 2})$ is 
equal to $[H_1]\in\check{H}^1(Z,\underline{\Z}^{n\choose 2})$.
\end{lemma}

Thus, if $H_1\neq 0$  the T-duality pair $(c,\delta)$ is not T-dual to a 
classical principal torus bundle. Instead, the comments in the previous section tell us that 
the bundle T-dual to $(c,\delta)$ is a noncommutative torus bundle. One thinks of the class 
$[H_2,H_1]$ as the Euler vector of this noncommutative torus bundle, with the cochains 
$H_1$ and $H_2$ describing the parameters of the noncommutative torus fibre and bundle ``twisting", respectively. 
Indeed, Lemma \ref{fibreagree} justifies the first statement, whilst the following lemma justifies 
the second (see also \cite{BHMa}):

\begin{lemma}[\cite{BouCarRat2}]
If $[H]\in\check{H}^3(X,\underline{\Z})|_{\pi^{0,3}=0}$ maps to a class of the form 
$[H_3,H_2,0]\in\mathbb{H}^3(\pi(\U),\underline{\Z})$, then the action $\alpha$ can 
be chose to have trivial Mackey obstruction. In this case, $CT(X,[H])\rtimes_\alpha\R^n$ 
can be viewed as the algebra of sections of a (commutative) principal $\T^n$-bundle that is 
classified by $[H_2]\in \check{H}^2(\pi(\U),\underline{\Z}^n)$.
\end{lemma}

\appendix

\section{Proof of Lemma \ref{cochainmap}}

\begin{lemma}\label{dczerodiffform}
Let $Z$ be $C^\infty$ manifold, $\W=\{W_{\lambda_0}\}$ an open cover of $Z$, and 
$A,B\in \check{Z}^2(\W,\underline{\Z})$ two cocycles. Let $C\in \check{C}^3(\W,\underline{\Z})$ 
be any 3-cochain that satisfies
\[A\cup B-B\cup A=\check\partial C.\]
Then the image of $\check\partial C$ in $\Omega^4(Z)$ under the collating formula is identically zero.
\end{lemma}
\begin{proof}
Follows from the cocycle identity for $A$ and $B$, the fact that the restriction of the 
image of $\check\partial C=A\cup B-B\cup A$ to the set $W_{\lambda_0}$ is given by
\begin{equation}\label{ABminusBA}
\sum_{\lambda_1\dots\lambda_4}\left(A_{\lambda_0\lambda_1\lambda_2}B_{\lambda_2\lambda_3
\lambda_4}-B_{\lambda_0\lambda_1\lambda_2}A_{\lambda_2\lambda_3\lambda_4}
\right)d\rho_{\lambda_1} \dots d\rho_{\lambda_4},
\end{equation}
where $\{\rho_{\lambda_0}\}$ is a partition of unity subordinate to $\W$, and the fact that
\[\sum_{\mu}(d\rho_{\mu})=d\left(\sum_{\mu}\rho_{\mu}\right)=d1=0\,.\]
\end{proof}

\begin{lemma}\label{maptoBHM}
Let $Z$ be $C^\infty$ manifold, $\W$ a good open cover of $Z$, and 
$F\in \check{Z}^2(\W,\underline{\Z}^n)$. Let the image of $F$ in $\Omega^2(Z,\t)$ 
under the \v{C}ech-de Rham isomorphism be $F_2$. Then the map
\[C_{F}^k(\W,\underline{\Z})\to C_{F_2}^{k,(0,2)}(Z,\t^*)\,, \]
defined by (see Section \ref{BHMwork})
\[(\phi^{k0},\phi^{(k-1)1},\phi^{(k-2)2})\mapsto(H_{k0},H_{(k-1)1},H_{(k-2)2})\,, \]
maps $D_F$-cocycles to $D_{F_2}$-cocycles.
\end{lemma}
\begin{proof}
Suppose that the triple $(\phi^{k0},\phi^{(k-1)1},\phi^{(k-2)2})$ is closed under 
$D_F$ and has image $[(H_{k0},H_{(k-1)1},H_{(k-2)2})$. 
Straight from the definitions one can see that $dH_{(k-2)2}=0$. Next, since 
the cocycle identity for $[\phi^{k0},\phi^{(k-1)1},\phi^{(k-2)2}]$ implies
\begin{align*}
H_{(k-1)1}=&\sum_{\lambda_1,\dots,\lambda_{k-1}}
\phi^{(k-1)1}_{\lambda_0\dots\lambda_{k-1}}(\cdot)_{l}d\rho_{\lambda_1}\dots d\rho_{\lambda_{k-1}}\\
&\quad+(-1)^{k+1}\sum_{\lambda_{1},\dots,\lambda_k}(\phi^{(k-2)2}\cup_1 F)_{\lambda_{0}
\dots\lambda_k}(\cdot)_l\rho_{\lambda_1}d\rho_{\lambda_2}\dots d\rho_{\lambda_k},
\end{align*}
we need to check if
\begin{align}
d(H_{(k-1)1})_l|_{W_{\lambda_0}}=&(-1)^{k+1}\sum_{\lambda_{1},\dots,\lambda_k}(\phi^{(k-2)2}
\cup_1 F)_{\lambda_{0}\lambda_1\dots\lambda_k}(\cdot)_ld\rho_{\lambda_1} 
d\rho_{\lambda_2}\dots d\rho_{\lambda_k} \notag\\
\label{diffmiddle}=& (-1)^{k+1}(H_{(k-2)2}\wedge F_2)_l|_{W_{\lambda_0}}.
\end{align}
By definition, $\phi^{(k-2)2}\cup_1 F$ is the \v{C}ech $k$-cocycle with values on 
$U^0_{\lambda_0\dots\lambda_{k}}$ in the $l^{th}$-component of $\Z^n$ given by
\[z\mapsto \sum_{1\leq i<j\leq n} \phi^{(k-2)2}_{\lambda_0\dots\lambda_{k-2}}(z)_{ij}
(F_{\lambda_{k-2}\lambda_{k-1}\lambda_k}(z)_i(e_l)_j-(e_l)_i
F_{\lambda_{k-2}\lambda_{k-1}\lambda_k}(z)_j).\]
When restricted to the set $W_{\lambda_0}$, this cocycle has $l^{th}$-component 
image the differential form:
\begin{equation}\label{H1LHS}
\sum_{\stackrel{1\leq i<j\leq n}{\lambda_1,\dots,\lambda_k}} \phi^{(k-2)2}_{\lambda_0
\dots\lambda_{k-2}}(z)_{ij}(F_{\lambda_{k-2}\lambda_{k-1}\lambda_k}(z)_i(e_l)_j-
(e_l)_iF_{\lambda_{k-2}\lambda_{k-1}\lambda_k}(z)_j)d\rho_{\lambda_1} 
d\rho_{\lambda_2}\dots d\rho_{\lambda_k}.
\end{equation}
On the other hand, the right hand side of Equation (\ref{diffmiddle}) is
\begin{align}
&(H_{(k-2)2}\wedge F_2)_l|_{W_{\lambda_0}}\notag\\
=&\sum_{1\leq i<j\leq n}(H_{(k-2)2})_{ij}|_{W_{\lambda_0}}\wedge\big[(F_2)_i
|_{W_{\lambda_0}}(e_l)_j-(F_2)_j |_{W_{\lambda_0}}(e_l)_i\big]\notag\\
\label{stage}=&\sum_{\stackrel{1\leq i<j\leq n}{\lambda_1,\dots,\lambda_k}}\phi^{(k-2)2}_{\lambda_0
\dots\lambda_{k-2}}(\cdot)_{ij}\left( F_{\lambda_0\lambda_{k-1}\lambda_k}(\cdot)_i(e_l)_j-
F_{\lambda_0\lambda_{k-1}\lambda_k}(\cdot)_j(e_l)_i\right)d\rho_{\lambda_1} \dots d\rho_{\lambda_k}.
\end{align}
To help us deal with Equation (\ref{stage}), we claim that
\begin{align*}
&\sum_{\stackrel{1\leq i<j\leq n}{\lambda_1,\dots,\lambda_k}}\phi^{(k-2)2}_{\lambda_0
\dots\lambda_{k-2}}(\cdot)_{ij}(F_{\lambda_0\lambda_{k-2}\lambda_k}-
F_{\lambda_0\lambda_{k-2}\lambda_{k-1}})(\cdot)_i(e_l)_j\\
&\quad -(F_{\lambda_0\lambda_{k-2}\lambda_k}-F_{\lambda_0\lambda_{k-2}\lambda_{k-2}})
(\cdot)_j(e_l)_id\rho_{\lambda_1} \dots  d\rho_{\lambda_k}=0.
\end{align*}
Indeed, if we expanded the above, each term $\phi^{(k-2)2}_{\lambda_0\dots\lambda_{k-2}}
(\cdot)_{ij}F_{\lambda_0\lambda_{k-2}\lambda_{\bullet}}$ would be missing either $\lambda_{k-1}$ 
or $\lambda_k$. Then our claim follows from the fact that
\[\sum_\mu d\rho_\mu=d\left(\sum_\mu \rho_\mu\right)=0.\]
Therefore, the fact that 
$F_{\lambda_0\lambda_{k-1}\lambda_k}=F_{\lambda_{k-2}\lambda_{k-1}\lambda_k}
+F_{\lambda_0\lambda_{k-2}\lambda_k}-F_{\lambda_0\lambda_{k-2}\lambda_{k-1}}$, 
implies Equation (\ref{stage}) is equal to
\[\sum_{\stackrel{1\leq i<j\leq n}{\lambda_1,\dots,\lambda_k}}\phi^{(k-2)2}_{\lambda_0
\dots \lambda_{k-2}}(\cdot)_{ij}\left( F_{\lambda_{k-2}\lambda_{k-1}\lambda_k}(\cdot)_i(e_l)_j-
F_{\lambda_{k-2}\lambda_{k-1}\lambda_k}(\cdot)_j(e_l)_i\right)d\rho_{\lambda_1} \dots 
d\rho_{\lambda_k}.\]
The above is exactly Equation (\ref{H1LHS}), and therefore the left and right hand sides of 
Equation (\ref{diffmiddle}) agree.

Lastly, since the cocycle identity for $[\phi^{k0},\phi^{(k-1)1},\phi^{(k-2)2}]$ implies
\begin{align*}
(H_{k0})|_{W_{\lambda_0}}=&\sum_{\lambda_1,\dots,\lambda_{k-1}}\phi^{k0}_{\lambda_0\dots
\lambda_{k-1}}(\cdot)d\rho_{\lambda_1}\dots d\rho_{\lambda_{k-1}}\\
&\quad +(-1)^{k+2}\sum_{\lambda_{1},\dots,\lambda_k}(\phi^{(k-1)1}\cup_1 F)_{\lambda_{0}
\dots\lambda_{k+1}}(\cdot)\rho_{\lambda_1}d\rho_{\lambda_2}\dots d\rho_{\lambda_{k+1}}\\
&\quad +(-1)^{k+2}\sum_{\lambda_{1},\dots,\lambda_k}(\phi^{(k-2)2}\cup_2 C(F))_{\lambda_{0}
\dots\lambda_{k+1}}(\cdot)\rho_{\lambda_1}d\rho_{\lambda_2}\dots d\rho_{\lambda_{k+1}}\\
&\quad +(-1)^{k+1}D(\phi^{(k-2)2}, C(F))|_{W_{\lambda_0}},
\end{align*}
we need to check if
\begin{align}
\label{cancelledterm}d(H_{k0})|_{W_{\lambda_0}}=&(-1)^{k+2}\sum_{\stackrel{\lambda_{1},
\dots,\lambda_k}{l=1,\dots, n}}(\phi^{(k-1)1}\cup_1 F)_{\lambda_{0}\dots\lambda_{k+1}}(\cdot)_l
d\rho_{\lambda_1}\dots d\rho_{\lambda_{k+1}}\\
&\quad +(-1)^{k+2}\sum_{\stackrel{1\leq i<j\leq n}{\lambda_{1},\dots,\lambda_k}}(\phi^{(k-2)2}\cup_2 
C(F))_{\lambda_{0}\dots\lambda_{k+1}}(\cdot)_{ij}d\rho_{\lambda_1}\dots d\rho_{\lambda_{k+1}}\notag\\
&\quad +(-1)^{k+1}dD(\phi^{(k-2)2},C(F))|_{W_{\lambda_0}}\notag\\
\label{difffirst}=& (-1)^{k+2}(H_{(k-1)1}\wedge F_2)|_{W_{\lambda_0}}.
\end{align}
We will prove this for $k\geq 3$, since the $k=2$ case is easier and uses similar techniques. Ignoring 
the factor of $(-1)^{k+2}$, the right hand side above is
\begin{align}
\label{line1} (H_{(k-1)1}\wedge F)|_{W_{\lambda_0}}:= & \ 
\sum_{l=1}^n\left(\left[\sum_{\lambda_1,\dots,\lambda_{k-1}}\phi^{(k-1)1}_{\lambda_0\dots
\lambda_{k-1}}(\cdot)_{l}d\rho_{\lambda_1}\dots d\rho_{\lambda_{k-1}}\right.\right.   \\
\label{line2}&\quad+(-1)^{k+1}\left.\sum_{\lambda_{1},\dots,\lambda_{k}}(\phi^{(k-2)2}\cup_1 F)_{\lambda_{0}
\dots\lambda_k}(\cdot)_l\rho_{\lambda_1}d\rho_{\lambda_2}\dots d\rho_{\lambda_k}\right]\\
\label{line3}&\quad\wedge \left.\sum_{\lambda_{k+1},\lambda_{k+2}}F_{\lambda_0\lambda_{k+1}
\lambda_{k+2}}(\cdot)_ld\rho_{\lambda_{k+1}} d\rho_{\lambda_{k+2}}\right).
\end{align}
We claim that the wedge product of the term on line (\ref{line2}) with the term on line (\ref{line3}) is zero. Indeed
\begin{align*}
\sum_{l=1}^n& \left[\left(  \sum_{\lambda_1,\dots,\lambda_k}(\phi^{(k-2)2}\cup_1 F)_{\lambda_{0}
\dots\lambda_k}(\cdot)_l\rho_{\lambda_1}d\rho_{\lambda_2}\dots d\rho_{\lambda_k}\right)\right. \\
& \quad\left.\wedge\sum_{\lambda_{k+1}\lambda_{k+2}}F_{\lambda_0\lambda_{k+1}
\lambda_{k+2}}(\cdot)_ld\rho_{\lambda_{k+1}} d\rho_{\lambda_{k+2}}\right] 
\end{align*}
\begin{align*}
=&\sum_{\stackrel{1\leq i<j\leq n}{\lambda_1,\dots,\lambda_{k+2}}}(\phi^{(k-2)2}_{\lambda_0
\dots\lambda_{k-2}}(\cdot)_{ij}(F_{\lambda_{k-2}\lambda_{k-1}\lambda_k}(\cdot)_iF_{\lambda_0
\lambda_{k+1}\lambda_{k+2}}(\cdot)_j\\
&\quad -F_{\lambda_{k-2}\lambda_{k-1}\lambda_k}(\cdot)_jF_{\lambda_0\lambda_{k+1}\lambda_{k+2}}(\cdot)_i)
\rho_{\lambda_1}d\rho_{\lambda_2}\dots d\rho_{\lambda_{k+2}}\\
=&\sum_{\stackrel{1\leq i<j\leq n}{\lambda_1,\dots,\lambda_{k+2}}}(\phi^{(k-2)2}_{\lambda_0
\dots\lambda_{k-2}}(\cdot)_{ij}(F_{\lambda_{0}\lambda_{k-1}\lambda_k}(\cdot)_iF_{\lambda_0
\lambda_{k+1}\lambda_{k+2}}(\cdot)_j\\
&\quad -F_{\lambda_{0}\lambda_{k-1}\lambda_k}(\cdot)_jF_{\lambda_0\lambda_{k+1}
\lambda_{k+2}}(\cdot)_i)\rho_{\lambda_1}d\rho_{\lambda_2}\dots d\rho_{\lambda_{k+2}}\\
=&0.
\end{align*}
Therefore, for $k\geq 3$, since the wedge product of the term on line (\ref{line1}) with the term on 
line (\ref{line3}) is exactly the term on line (\ref{cancelledterm}), we just need to show 
$dD(\phi^{(k-2)2},C(F))|_{W_{\lambda_0}}$ is equal to
\[\sum_{\stackrel{1\leq i<j\leq n}{\lambda_{1},\dots,\lambda_{k+1}}}(\phi^{(k-2)2}\cup_2 
C(F))_{\lambda_{0}\dots\lambda_{k+1}}(\cdot)_{ij}d\rho_{\lambda_1}\dots d\rho_{\lambda_{k+1}}.\]
First, note Lemma \ref{dczerodiffform} implies $\check\partial C(F)_{ij}$ is the zero form. 
Then, the facts that $\check\partial$ is a graded derivation with respect to $\cup_2$ 
and $\phi^{(k-2)2}$ is $\check\partial$-closed imply
\[\sum_{\stackrel{1\leq i<j\leq n}{\lambda_{1},\dots,\lambda_{k+2}}}\check\partial(\phi^{(k-2)2}\cup_2 
C(F))_{\lambda_{0}\dots\lambda_{k+2}}(\cdot)_{ij}\rho_{\lambda_1}d\rho_{\lambda_2}\dots d\rho_{\lambda_{k+2}}=0.\]
Therefore
\begin{align*}
&\sum_{\stackrel{1\leq i<j\leq n}{\lambda_{1},\dots,\lambda_{k+1}}}(\phi^{(k-2)2}\cup_2 
C(F))_{\lambda_{0}\dots\lambda_{k+1}}d\rho_{\lambda_1}d\rho_{\lambda_2}\dots d\rho_{\lambda_{k+1}}\\
&\quad +\sum_{\stackrel{1\leq i<j\leq n}{\lambda_{1},\dots,\lambda_{k+2}}}\check\partial(\phi^{(k-2)2}\cup_2 
C(F))_{\lambda_{1}\dots\lambda_{k+2}}\rho_{\lambda_1}d\rho_{\lambda_2}\dots d\rho_{\lambda_{k+2}}
\end{align*}
\begin{align*}
=&\sum_{\stackrel{1\leq i<j\leq n}{\lambda_{1},\dots,\lambda_{k+1}}}(\phi^{(k-2)2}\cup_2 
C(F))_{\lambda_{0}\dots\lambda_{k+1}}d\rho_{\lambda_1}d\rho_{\lambda_2}\dots d\rho_{\lambda_{k+1}}\\
&\quad +\sum_{\stackrel{1\leq i<j\leq n}{\lambda_{1},\dots,\lambda_{k+2}}}((\check\partial\phi^{(k-2)2})\cup_2 
C(F))_{\lambda_{1}\dots\lambda_{k+2}}\rho_{\lambda_1}d\rho_{\lambda_2}\dots d\rho_{\lambda_{k+2}}\\
=&\sum_{\stackrel{1\leq i<j\leq n}{\lambda_{1},\dots,\lambda_{k+2}}}(\phi^{(k-2)2}\cup_2 
C(F))_{\lambda_{1}\dots\lambda_{k+2}}\rho_{\lambda_1}d\rho_{\lambda_2}\dots d\rho_{\lambda_{k+2}}.
\end{align*}
Now, observe that swapping indices $\lambda_k\leftrightarrow\lambda_{k+2}$ and 
$\lambda_{k-1}\leftrightarrow\lambda_{k+1}$ shows
\begin{align*}
&\sum_{\stackrel{1\leq i<j\leq n}{\lambda_1,\dots,\lambda_{k+2}}}\phi_{\lambda_1\dots
\lambda_{k-1}}^{(k-2)2}(\cdot)_{ij}(F_{\lambda_{k}\lambda_{k+1}\lambda_{k+2}}(\cdot)_{i}
F_{\lambda_{k-1}\lambda_{k}\lambda_{k+2}}(\cdot)_{j})\rho_{\lambda_1}d\rho_{\lambda_2}\dots 
d\rho_{\lambda_{k+2}}=0,\end{align*}
and
\begin{align*}
&\sum_{\stackrel{1\leq i<j\leq n}{\lambda_1,\dots,\lambda_{k+2}}}\phi_{\lambda_{k}\lambda_1
\dots\lambda_{k-2}}^{(k-2)2}(\cdot)_{ij}(F_{\lambda_{k-1}\lambda_{k}\lambda_{k+1}}(\cdot)_{i}
F_{\lambda_{k-1}\lambda_{k+1}\lambda_{k+2}}(\cdot)_{j})\rho_{\lambda_1}d\rho_{\lambda_2}\dots d\rho_{\lambda_{k+2}}=0.
\end{align*}
Thus, using these two identities we see
\begin{align*}
\sum_{\stackrel{1\leq i<j\leq n}{\lambda_1,\dots,\lambda_{k+2}}}& \phi_{\lambda_1
\dots\lambda_{k-1}}^{(k-2)2}(\cdot)_{ij}C(F)_{\lambda_{k-1}\dots \lambda_{k+2}}\rho_{\lambda_1}
d\rho_{\lambda_2}\dots d\rho_{\lambda_{k+2}}\\
=&\sum_{\stackrel{1\leq i<j\leq n}{\lambda_1,\dots,\lambda_{k+2}}}\phi_{\lambda_1
\dots\lambda_{k-1}}^{(k-2)2}(\cdot)_{ij}(F_{\lambda_{k-1}\lambda_{k}\lambda_{k+1}}(\cdot)_{i}
F_{\lambda_{k-1}\lambda_{k+1}\lambda_{k+2}}(\cdot)_{j})\rho_{\lambda_1}d\rho_{\lambda_2}
\dots d\rho_{\lambda_{k+2}}\\
=&\sum_{\stackrel{1\leq i<j\leq n}{\lambda_1,\dots,\lambda_{k+2}}}((-1)^k\phi_{\lambda_{k}\lambda_1
\dots\lambda_{k-2}}^{(k-2)2}+\phi_{\lambda_k\lambda_2\dots\lambda_{k-1}}^{(k-2)2})(\cdot)_{ij}\\
&\quad \times(F_{\lambda_{k-1}\lambda_{k}\lambda_{k+1}}(\cdot)_{i}F_{\lambda_{k-1}\lambda_{k+1}
\lambda_{k+2}}(\cdot)_{j})\rho_{\lambda_1}d\rho_{\lambda_2}\dots d\rho_{\lambda_{k+2}}\\
=&\sum_{\stackrel{1\leq i<j\leq n}{\lambda_2,\dots,\lambda_{k+2}}}(\phi_{\lambda_k\lambda_2
\dots\lambda_{k-1}}^{(k-2)2})(\cdot)_{ij}(F_{\lambda_{k-1}\lambda_{k}\lambda_{k+1}}(\cdot)_{i}
F_{\lambda_{k-1}\lambda_{k+1}\lambda_{k+2}}(\cdot)_{j})d\rho_{\lambda_2}\dots d\rho_{\lambda_{k+2}}\\
=&d\left(\sum_{\stackrel{1\leq i<j\leq n}{\lambda_1,\dots,\lambda_{k+1}}}\phi_{\lambda_{k-1}\lambda_1
\dots\lambda_{k-2}}^{(k-2)2}(\cdot)_{ij}(F_{\lambda_{k-2}\lambda_{k-1}\lambda_{k}}(\cdot)_{i}
F_{\lambda_{k-2}\lambda_{k}\lambda_{k+1}}(\cdot)_{j})\rho_{\lambda_1}d\rho_{\lambda_2}\dots d\rho_{\lambda_{k+1}}\right)\\
=&dD(\phi^{(k-2)2},C(F))|_{W_{\lambda_0}}.
\end{align*}
This completes the verification of Equation (\ref{difffirst}).
\end{proof}

\begin{lemma}\label{coboundarymaptoBHM}
Let $Z$ be a $C^\infty$ manifold, $\W$ a good open cover of $Z$, and $F\in \check{Z}^2(\W,\underline{\Z}^n)$. 
Let the image of $F$ in $\Omega^2(Z,\t)$ under the \v{C}ech-de Rham isomorphism be $F_2$. Then the map
\[C_{F}^k(\W,\underline{\Z})\to C_{F_2}^{k,(0,2)}(Z,\t^*)\,, \]
defined by
\[(\phi^{k0},\phi^{(k-1)1},\phi^{(k-2)2})\mapsto(H_{k0},H_{(k-1)1},H_{(k-2)2})\,, \]
maps $D_F$-coboundaries to $D_{F_2}$-coboundaries.
\end{lemma}
\begin{proof}
We only prove this for $k\geq 3$, since the case $k\leq 2$ is similar and easier. Suppose
\begin{equation}\label{phidifferential}
(\phi^{k0},\phi^{(k-1)1},\phi^{(k-2)2})=D_F(\phi^{(k-1)0},\phi^{(k-2)1},\phi^{(k-3)2}).
\end{equation}
We claim there exists a cochain $(H_{(k-1)0}+(-1)^{k-1}H(C)_{(k-1)0},H_{(k-2)1},H_{(k-3)2})$ such that
\[(H_{k0},H_{(k-1)1},H_{(k-2)2})=D_{F_2}(H_{(k-1)0}+(-1)^{k-1}H(C)_{(k-1)0},H_{(k-2)1},H_{(k-3)2}).\]
Indeed, the formulas for $H_{(k-1)0}$, $H_{(k-2)1}$, $H_{(k-3)2}$, and $H(C)_{(k-1)0}$ are
\begin{align*}
H_{(k-3)2}|_{W_{\lambda_0}}:=&\sum_{\lambda_{1},\dots, \lambda_{k-2}}\phi^{(k-3)2}_{\lambda_1
\dots \lambda_{k-2}}\rho_{\lambda_1}d\rho_{\lambda_2}\dots d\rho_{\lambda_{k-2}} \,, \\
H_{(k-2)1}|_{W_{\lambda_0}}:=&\sum_{\lambda_{1},\dots,\lambda_{k-1}}\phi^{(k-2)1}_{\lambda_1
\dots \lambda_{k-1}}(\cdot)_l\rho_{\lambda_1}d\rho_{\lambda_2}\dots \lambda_{k-1}\,, \\
H_{(k-1)0}|_{W_{\lambda_0}}:=&\sum_{\lambda_{1},\dots,\lambda_{k}}\phi^{(k-1)0}_{\lambda_{1}
\dots\lambda_{k}}\rho_{\lambda_2}d\rho_{\lambda_3}\dots d\rho_{\lambda_{k+1}} \,,
\end{align*}
and
\begin{align*}
&H(C)_{(k-1)0}|_{W_{\lambda_0}}:=\\
&\begin{cases}
-\sum_{\stackrel{1\leq i<j\leq n}{\lambda_{1},\dots,\lambda_{4}}}\phi_{2}^{02}(\cdot)_{ij}
C_{\lambda_{1}\dots\lambda_{4}}(F)_{ij}\rho_{\lambda_1}\rho_{\lambda_2}d\rho_{\lambda_3}
d\rho_{\lambda_{4}}& k=3\\
-\sum_{\stackrel{1\leq i<j\leq n}{\lambda_{1},\dots,\lambda_{k+1}}}\phi_{\lambda_{k-1}\lambda_2
\dots\lambda_{k-2}\lambda_{k-1}}^{(k-3)2}(\cdot)_{ij}C(F)_{\lambda_{k-2}
\dots\lambda_{k+1}}(\cdot)_{ij}\rho_{\lambda_2}\dots d\rho_{\lambda_{k+1}}& k>3.
\end{cases}
\end{align*}
Now, if $\phi^{(k-2)2}=\check\partial \phi^{(k-3)2}$ we have
\begin{align*}
(H_{(k-2)2})_{ij}|_{W_{\lambda_0}}:=&\sum_{\lambda_1,\dots, \lambda_{k-2}}
\phi^{(k-2)2}_{\lambda_0\dots\lambda_{k-2}}(\cdot)_{ij}d\rho_{\lambda_1}\dots d\rho_{\lambda_{k-2}}\\
=&\sum_{\lambda_1,\dots, \lambda_{k-2}}\check\partial \phi^{(k-3)2}_{\lambda_0
\dots\lambda_{k-2}}(\cdot)_{ij}d\rho_{\lambda_1}\dots d\rho_{\lambda_{k-2}}\\
=&d\left(\sum_{\lambda_1,\dots, \lambda_{k-2}}\phi^{(k-3)2}_{\lambda_1\dots\lambda_{k-2}}(\cdot)_{ij}
\rho_{\lambda_1}\dots d\rho_{\lambda_{k-2}}\right)\\
=&dH_{(k-3)2}|_{W_{\lambda_0}}.
\end{align*}
Also, the Equation (\ref{phidifferential}) implies
\begin{align*}
(H_{(k-1)1})_{l}|_{W_{\lambda_0}}:= &\sum_{\lambda_{1},\dots,\lambda_k}(\check\partial
\phi^{(k-2)1}+(-1)^{k-1}\phi^{(k-3)2}\cup F)_{\lambda_{1}\dots\lambda_k}(\cdot)_{l}\rho_{\lambda_1}
d\rho_{\lambda_2}\dots d\rho_{\lambda_{k}}
\end{align*}
\begin{align*}
=&d\left(\sum_{\lambda_{2},\dots,\lambda_k}\phi^{(k-2)1}_{\lambda_{2}\dots\lambda_k}(\cdot)_{l}\rho_{\lambda_2}
d\rho_{\lambda_{3}}\dots d\rho_{\lambda_{k}}\right)\\
&\quad+(-1)^{k-1}\left(\sum_{\lambda_{1},\dots, \lambda_{k-2}}\phi^{(k-3)2}_{\lambda_1\dots \lambda_{k-2}}
\rho_{\lambda_1}d\rho_{\lambda_2}\dots d\rho_{\lambda_{k-2}}\right)\\
&\quad\wedge\left( \sum_{\lambda_{k-1},\lambda_k}F_{\lambda_{0}\lambda_{k-1}\lambda_k}(\cdot)_{l}
d\rho_{\lambda_{k-1}}d\rho_{\lambda_{k}}\right)\\
=&dH_{(k-2)1}|_{W_{\lambda_0}}+(-1)^{k-1}H_{(k-3)2}|_{W_{\lambda_0}}\wedge F_2|_{W_{\lambda_0}}.
\end{align*}
Now we deal with the last term. Observe that Equation (\ref{phidifferential}) implies
\begin{align*}
(H_{k0})|_{W_{\lambda_0}}:=&\sum_{\lambda_{1},\dots,\lambda_{k+1}}\phi^{k0}_{\lambda_{1}
\dots\lambda_{k+1}}\rho_{\lambda_1}d\rho_{\lambda_2}\dots d\rho_{\lambda_{k+1}}+ (-1)^{k+2}D(\phi^{(k-2)2},C(F))\\
=&\sum_{\lambda_{1},\dots,\lambda_{k+1}}\left(\check\partial\phi^{(k-1)0}+(-1)^{k}\phi^{(k-2)1}\cup_1 F\right.\\
&\quad +\left.(-1)^{k}\phi^{(k-3)2}\cup_2 C(F))_{\lambda_{1}\dots\lambda_{k+1}}\rho_{\lambda_1}
d\rho_{\lambda_2}\dots d\rho_{\lambda_{k+1}}\right)\\
&\quad +(-1)^{k+1}D(\check\partial\phi^{(k-3)2},C(F))|_{W_{\lambda_0}}\\
=&d\left(\sum_{\lambda_{2},\dots,\lambda_{k+1}}\phi^{(k-1)0}_{\lambda_{2}\dots\lambda_{k+1}}
\rho_{\lambda_2}d\rho_{\lambda_3}\dots d\rho_{\lambda_{k+1}}\right)\\
&\quad+(-1)^{k}\sum_{l=1}^n\left[\left(\sum_{\lambda_{1}\dots\lambda_{k-1}}\phi^{(k-2)1}(\cdot)_l
\rho_{\lambda_1}d\rho_{\lambda_2}\dots \lambda_{k-1}\right)\right.\\
&\quad\quad\wedge\left.\left( \sum_{\lambda_{k},\lambda_{k+1}}F_{\lambda_{0}\lambda_{k}\lambda_{k+1}}(\cdot)_l
d\rho_{\lambda_k} d\rho_{\lambda_{k+1}}\right)\right]\\
&\quad+(-1)^{k}(\phi^{(k-3)2}\cup_2 C(F))_{\lambda_{1}\dots\lambda_{k+1}}\rho_{\lambda_1}d\rho_{\lambda_2}\dots 
d\rho_{\lambda_{k+1}}\\
&\quad+(-1)^{k-1}D(\check\partial\phi^{(k-3)2},C(F))|_{W_{\lambda_0}}\\
=&dH_{(k-1)0}|_{W_{\lambda_0}}+(-1)^{k}H_{(k-2)1}|_{W_{\lambda_0}}\wedge F_2|_{W_{\lambda_0}}\\
&\quad+(-1)^{k}(\phi^{(k-3)2}\cup_2 C(F))_{\lambda_{1}\dots\lambda_{k+1}}\rho_{\lambda_1}
d\rho_{\lambda_2}\dots d\rho_{\lambda_{k+1}}\\
&\quad+(-1)^{k-1}D(\check\partial\phi^{(k-3)2},C(F))|_{W_{\lambda_0}}.
\end{align*}
Removing the factor of $(-1)^{k-1}$, it therefore suffices to show that 
$D(\check\partial\phi^{(k-3)2},C(F))|_{W_{\lambda_0}}$ differs from
\begin{equation}\label{phi02cupc}
\sum_{\stackrel{1\leq i<j\leq n}{\lambda_{1}\dots\lambda_{k+1}}}(\phi^{(k-3)2}\cup_2 
C(F))_{\lambda_{1}\dots\lambda_{k+1}}(\cdot)_{ij}\rho_{\lambda_1}d\rho_{\lambda_2}\dots d\rho_{\lambda_{k+1}}\,,
\end{equation}
by $dH(C)_{(k-1)0}|_{W_{\lambda_0}}$.

We provide the details for the case $k=3$; the cases $k>3$ are similar (the reader can consult 
\cite{Rat} for the details). When $k=3$ we have
\begin{align}
D(\check\partial\phi^{02},C(F))|_{W_{\lambda_0}}=&
\sum_{\stackrel{1\leq i<j\leq n}{\lambda_{1},\dots,\lambda_{4}}}(\phi_{\lambda_{1}}^{02}-
\phi_{\lambda_{2}}^{02})_{ij}C(F)_{\lambda_{1}\dots\lambda_{4}}(\cdot)_{ij}\rho_{\lambda_1}
d\rho_{\lambda_2}\dots d\rho_{\lambda_{4}}\notag\\
\label{phi02cupc2}=&\sum_{\stackrel{1\leq i<j\leq n}{\lambda_{1},\dots,\lambda_{4}}}
\phi_{\lambda_{1}}^{02}(\cdot)_{ij}C(F)_{\lambda_{1}\dots\lambda_{4}}(\cdot)_{ij}
\rho_{\lambda_1}d\rho_{\lambda_2}\dots d\rho_{\lambda_{4}} \\
\label{phi02cupcother}&\quad - \sum_{\stackrel{1\leq i<j\leq n}{\lambda_{1},\dots,
\lambda_{4}}}\phi_{\lambda_{2}}^{02}(\cdot)_{ij}C(F)_{\lambda_{1}\dots\lambda_{4}}(\cdot)_{ij}
\rho_{\lambda_1}d\rho_{\lambda_2}\dots d\rho_{\lambda_{4}}.
\end{align}
Observe that (\ref{phi02cupc2}) above is exactly (\ref{phi02cupc}), so we only need to show 
(\ref{phi02cupcother}) is equal to $dH(C)_{(k-1)0}|_{W_{\lambda_0}}$. Now, swapping 
indices $\lambda_2\leftrightarrow\lambda_3$ below shows
\[\sum_{\stackrel{1\leq i<j\leq n}{\lambda_{1},\dots,\lambda_{4}}}\phi_{\lambda_1}^{02}(\cdot)_{ij}
F_{\lambda_{2}\lambda_1\lambda_{3}}(\cdot)_iF_{\lambda_{2}\lambda_3\lambda_{4}}(\cdot)_j
\rho_{\lambda_1}d\rho_{\lambda_2}\dots d\rho_{\lambda_{4}}=0.\]
Now we compute, interchanging indices $\lambda_1$ and $\lambda_2$ from (\ref{onetwoswap1}) 
to (\ref{onetwoswap2}) below:
\begin{align}
&d\left(\sum_{\stackrel{1\leq i<j\leq n}{\lambda_{1},\dots,\lambda_{4}}}\phi_{\lambda_2}^{02}(\cdot)_{ij}
C(F)_{\lambda_{1}\dots\lambda_{4}}(\cdot)_{ij}\rho_{\lambda_1}\rho_{\lambda_2}d\rho_{\lambda_3}
d\rho_{\lambda_{4}}\right)\notag\\
=&\sum_{\stackrel{1\leq i<j\leq n}{\lambda_{1},\dots,\lambda_{4}}}\phi_{2}^{02}(\cdot)_{ij}
C(F)_{\lambda_{1}\dots\lambda_{4}}(\cdot)_{ij}\rho_{\lambda_1}d\rho_{\lambda_2}\dots d\rho_{\lambda_{4}}\notag\\
\label{onetwoswap1}&\quad\sum_{\lambda_{1},\dots,\lambda_{4}}\phi_{\lambda_2}^{02}(\cdot)_{ij}
C(F)_{\lambda_{1}\dots\lambda_{4}}(\cdot)_{ij}\rho_{\lambda_2}d\rho_{\lambda_1}d\rho_{\lambda_3}d\rho_{\lambda_{4}}
\end{align}
\begin{align}
=&\sum_{\stackrel{1\leq i<j\leq n}{\lambda_{1},\dots,\lambda_{4}}}\phi_{\lambda_2}^{02}(\cdot)_{ij}
C(F)_{\lambda_{1}\dots\lambda_{4}}(\cdot)_{ij}\rho_{\lambda_1}d\rho_{\lambda_2}\dots d\rho_{\lambda_{4}}\notag\\
\label{onetwoswap2}&\quad +\sum_{\lambda_{1},\dots,\lambda_{4}}\phi_{\lambda_1}^{02}(\cdot)_{ij}
C(F)_{\lambda_{2}\lambda_{1}\lambda_{3}\lambda_{4}}(\cdot)_{ij}\rho_{\lambda_1}d\rho_{\lambda_2}\dots d\rho_{\lambda_{4}}\\
=&\sum_{\stackrel{1\leq i<j\leq n}{\lambda_{1},\dots,\lambda_{4}}}\phi_{\lambda_2}^{02}(\cdot)_{ij}
C(F)_{\lambda_{1}\dots\lambda_{4}}(\cdot)_{ij}\rho_{\lambda_1}d\rho_{\lambda_2}\dots d\rho_{\lambda_{4}}\notag\\
&\quad +\sum_{\lambda_{1},\dots,\lambda_{4}}\phi_{\lambda_1}^{02}(\cdot)_{ij}F_{\lambda_{2}\lambda_1\lambda_{3}}(\cdot)_i
F_{\lambda_{2}\lambda_3\lambda_{4}}(\cdot)_j\rho_{\lambda_1}d\rho_{\lambda_2}\dots d\rho_{\lambda_{4}}\notag\\
=&\sum_{\stackrel{1\leq i<j\leq n}{\lambda_{1},\dots,\lambda_{4}}}\phi_{\lambda_2}^{02}
(\cdot)_{ij}C(F)_{\lambda_{1}\dots\lambda_{4}}(\cdot)_{ij}\rho_{\lambda_1}
d\rho_{\lambda_2}\dots d\rho_{\lambda_{4}}.\notag
\end{align}
This last line is exactly $-1\times$(\ref{phi02cupcother}). Therefore $D(\check\partial\phi^{02},
C(F))|_{W_{\lambda_0}}$ differs from \[\sum_{\stackrel{1\leq i<j\leq n}{\lambda_{1}
\dots\lambda_{4}}}(\phi^{02}\cup_2 C(F))_{\lambda_{1},\dots,\lambda_{4}}(\cdot)_{ij}
\rho_{\lambda_1}d\rho_{\lambda_2}\dots d\rho_{\lambda_{4}}\]
by
\[d\left(-\sum_{\stackrel{1\leq i<j\leq n}{\lambda_{1},\dots,\lambda_{4}}}\phi_{2}^{02}
(\cdot)_{ij}C_{\lambda_{1}\dots\lambda_{4}}(F)_{ij}\rho_{\lambda_1}\rho_{\lambda_2}
d\rho_{\lambda_3}d\rho_{\lambda_{4}}\right)=dH(C)_{(k-1)0}|_{W_{\lambda_0}}.\]
Therefore we have shown, that if
\[(\phi^{k0},\phi^{(k-1)1},\phi^{(k-2)2})=D_F(\phi^{(k-1)0},\phi^{(k-2)1},\phi^{(k-3)2})\,,\]
then the image of $(\phi^{k0},\phi^{(k-1)1},\phi^{(k-2)2})$ is
\[(H_{k0},H_{(k-1)1},H_{(k-2)2})=D_{F_2}(H_{(k-1)0}+(-1)^{k-1}H(C)_{(k-1)0},H_{(k-2)1},H_{(k-3)2})]\,.\]
\end{proof}



\end{document}